\documentclass[11pt]{amsart}
\usepackage[utf8]{inputenc}
\usepackage{amsmath, amsthm, amssymb}
\usepackage[nobysame,abbrev,alphabetic]{amsrefs}
\usepackage{graphicx}
\usepackage[usenames]{color}
\usepackage{srcltx} % JD - to jump between source and preview
\usepackage{tikz-cd}
\usepackage{tikz}
\usetikzlibrary{positioning}
\usepackage{mathtools}

\newtheorem{thm}{Theorem}[section]
\newcommand{\bt}{\begin{thm}}
\newcommand{\et}{\end{thm}}

\newtheorem{conj}[thm]{Conjecture}

\newtheorem{ex}[thm]{Example}

\newtheorem{cor}[thm]{Corollary}   %remember switch all {coro} to {cor}
\newcommand{\bc}{\begin{cor}}
\newcommand{\ec}{\end{cor}}

\newtheorem{lem}[thm]{Lemma}   %remember to switch all {lem} to {lem}
\newcommand{\bl}{\begin{lem}}
\newcommand{\el}{\end{lem}}

\newtheorem{prop}[thm]{Proposition}
\newcommand{\bp}{\begin{prop}}
\newcommand{\ep}{\end{prop}}

\newtheorem{defn}[thm]{Definition}
\newcommand{\bd}{\begin{defn}}    % This produces an error????    
\newcommand{\ed}{\end{defn}}

\newtheorem{rmrk}[thm]{Remark}   %remember to switch all {rmrk} to {rmrk}

\newcommand{\br}{\begin{rmrk}}
\newcommand{\er}{\end{rmrk}}

\newcommand{\N}{\mathbb{N}}

\newcommand{\R}{\mathbb{R}}

\def\set{\textrm{set}}
\newcommand{\intcurr}{{\mathbf I}} 

\DeclareMathOperator{\Vol}{Vol}
\DeclareMathOperator{\Diam}{Diam}
\DeclareMathOperator{\Area}{Area}

\newcommand{\mass}{{\mathbf M}}
\newcommand{\diam}{\operatorname{Diam}}

\newcommand{\intcur}{{\mathbf I}} 

\newcommand{\lip}{\operatorname{Lip}}

\def\XXint#1#2#3{{\setbox0=\hbox{$#1{#2#3}{\int}$ }
\vcenter{\hbox{$#2#3$ }}\kern-.6\wd0}}

\begin{document}

\title[Null Distance and Convergence of Static Spacetimes]{Metric Convergence of Sequences of Static Spacetimes with the Null Distance}

\author{Brian Allen}
\address{CUNY Graduate Center and Lehman College}
%\urladdr{\url{https://sites.google.com/view/brian-allen}}
%\date{}

\begin{abstract}
How should one define metric space notions of convergence for sequences of spacetimes? Since a Lorentzian manifold does not define a metric space directly, the uniform convergence, Gromov-Hausdorff (GH) convergence, and Sormani-Wenger Intrinsic Flat (SWIF) convergence does not extend automatically. 
One approach is to define a metric space structure, which is compatible with the Lorentzian structure, so that the usual notions of convergence apply. This approach was taken by C. Sormani and C. Vega \cite{SV} when defining the null distance. In this paper, we study sequences of static spacetimes equipped with the null distance under uniform, GH, and SWIF convergence, as well as H\"{o}lder bounds. We use the results of the Volume Above Distance Below (VADB) theorem of the author, R. Perales, and C. Sormani \cite{Allen-Perales-Sormani} to prove an analog of the VADB theorem for sequences of static spacetimes with the null distance. We also give a conjecture of what the VADB theorem should be in the case of sequences of globally hyperbolic spacetimes with the null distance.
\end{abstract}

\maketitle

\section{Introduction}

How should one define metric space notions of convergence for sequences of spacetimes? Since a Lorentzian manifold does not define a metric space directly, the uniform convergence, Gromov-Hausdorff (GH) convergence, and Sormani-Wenger Intrinsic Flat (SWIF) convergence does not extend automatically. 
One approach is to define a metric space structure, which is compatible with the Lorentzian structure, so that the usual notions of convergence apply. This approach was taken up by C. Sormani and C. Vega \cite{SV} by defining the null distance. Notions of spacetime intrinsic flat convergence employing the null distance have been further developed by A. Sakovich and C. Sormani \cite{SS2}, pointing out important nuances which need to be appreciated when extending metric notions of convergence to spacetimes with the null distance. The problem of defining Gromov-Hausdorff convergence for sequences of spacetimes was initially taken up by J. Noldus \cite{N}, and has been more recently addressed by E. Minguzzi and S. Suhr \cite{MS, MS-LorentzCompactness}, and by O. M\"{u}ller \cite{M}. Recently, measured Lorentzian Gromov-Hausdorff convergence has been studied by M. Braun and C. S\"{a}mann \cite{BraunSamann}. These authors take the alternative approach of redefining GH convergence with the spacetime structure in mind instead of defining a metric space structure on the spacetime. It would be great to see examples computed which specifically compare and contrast these different approaches to spacetime metric convergence. 

The first exploration of convergence of sequences of spacetimes with the null distance was initiated by the author and A. Burtscher \cite{AB} in the case of warped products whose warping functions converge uniformly. Later, the author \cite{Allen-Null} extended this work to optimal conditions on a sequence of warping functions which imply uniform and GH convergence. Given the lack of tools for estimating the SWIF distance between spacetimes with the null distance, SWIF convergence was not handled in these previous works. In this paper, we give the first tools for estimating the SWIF distance between globally hyperbolic spacetimes equipped with the null distance (See section \ref{sect-Estimating SWIF}), using ideas first introduced by the author, R. Perales, and C. Sormani \cite{Allen-Perales-Sormani} for estimating the SWIF distance between Riemannian manifolds. Furthermore, we are able to use tools developed by the author and R. Perales \cite{Allen-Perales} and the author and E. Bryden \cite{Allen-Bryden} to prove a SWIF convergence result (See Theorem \ref{thm-MainTheorem SWIF}) under optimal conditions on a sequence of static spacetimes equipped with the null distance. By taking advantage of the work of the author \cite{Allen-Holder} for proving H\"{o}lder bounds for sequences of Riemannian manifolds, we are also able to provide optimal conditions which establish H\"{o}lder bounds for a sequence of static spacetimes with the null distance which imply uniform and GH convergence of a subsequence.

\begin{thm}\label{thm-MainTheorem Holder}
Let $0<C, C'<\infty$, $M$ a compact, connected, and oriented manifold, $(M,\sigma_1)$ a continuous Riemannian manifold, $(M,\sigma_0)$ a smooth Riemannian manifold, $L=[t_0,t_1]\times M$, $h_i:M \rightarrow (0,\infty)$, $g_i=-h_i^2dt^2+\sigma_j$, $i \in \{0,1\}$. If $p > n$,
\begin{align}
\int_M\left|\frac{\sigma_1}{h_1}\right|_{\sigma_{0}}^{\frac{p}{2}}dV_{\sigma_{0}} &\le C,\label{eq-L^pBound}
\end{align}
 then  then for all $x,y \in L$ we find
\begin{align}\
    \hat{d}_{t,g_1}(x,y) \le C' \hat{d}_{t,g_{0}}(x,y)^{\frac{p-n}{p}}.
\end{align}
\end{thm}

By applying the observation of Theorem \ref{thm-MainTheorem Holder} to a sequence of static spacetimes with additional necessary hypotheses we are able to imply uniform and GH convergence.

 \begin{thm}\label{thm-MainTheorem GH}
Let $0<A,C, C'<\infty$, $M$ a compact, connected, oriented manifold, $(M,\sigma_j)$ be a sequence of continuous Riemannian manifolds, $(M,\sigma_{\infty})$ a smooth Riemannian manifold,  $L=[t_0,t_1]\times M$, $h_j:M \rightarrow (0,\infty)$, $g_j=-h_j^2dt^2+\sigma_j$, $j \in \N\cup\{\infty\}$, with $h_j$ continuous for $j \in \N$ and smooth for $j =\infty$. If $p > n$,
\begin{align}
\int_M\left|\frac{\sigma_j}{h_j}\right|_{\sigma_{\infty}}^{\frac{p}{2}}dV_{\sigma_{\infty}} &\le C,\label{eq-L^pBound}
\\ \frac{\sigma_j(v,v)}{h_j^2} &\ge  \left( 1 -\frac{1}{j} \right)\frac{\sigma_{\infty}(v,v)}{h_{\infty}^2}, \quad \forall x \in M, v \in T_xM \label{eq-Space Component lower bound}
\\\Vol\left(M,\frac{\sigma_j}{h_j}\right)&\rightarrow \Vol\left(M,\frac{\sigma_{\infty}}{h_{\infty}}\right),\label{eq-VolumeBound}
\\ \Area\left(\partial M,\frac{\sigma_j}{h_j}\right) &\le A,\label{eq-AreaBound}
\end{align}
then  then for all $x,y \in L$ we find
\begin{align}\
    \hat{d}_{t,\tilde{g}_j}(x,y) \le C' \hat{d}_{t,g_{\infty}}(x,y)^{\frac{p-n}{p}},
\end{align}
and
\begin{align}
    (L,\hat{d}_{t,g_j})\rightarrow (L,\hat{d}_{t,g_{\infty}}),
\end{align}
in the uniform and Gromov-Hausdorff sense. 
\end{thm}
\begin{rmrk}
    It is natural to want to replace conditions like \eqref{eq-L^pBound}, \eqref{eq-VolumeBound}, and \eqref{eq-AreaBound} with conditions on $h_j$ and $\sigma_j$ separately. In this remark we discuss several such sufficient conditions which will also apply to the theorem and conjecture which follows. By H\"{o}lder's inequality with $\frac{1}{q_1}+\frac{1}{q_2}=1$ we know that
    \begin{align}
       \int_M\left|\frac{\sigma_j}{h_j}\right|_{\sigma_{\infty}}^{\frac{p}{2}}dV_{\sigma_{\infty}} \le \left(\int_M\left|\sigma_j\right|_{\sigma_{\infty}}^{\frac{pq_1}{2}}dV_{\sigma_{\infty}}\right)^{\frac{1}{q_1}}\left(\int_M\left|h_j\right|_{\sigma_{\infty}}^{-\frac{pq_2}{2}}dV_{\sigma_{\infty}}\right)^{\frac{1}{q_2}},
    \end{align}
    which clears up what is sufficient for \eqref{eq-L^pBound}. Now we also observe by the determinant trace inequality that
    \begin{align}
        \Vol\left(M,\frac{\sigma_j}{h_j}\right)&\le  \int_M\left|\frac{\sigma_j}{h_j}\right|_{\sigma_{\infty}}^{\frac{n}{2}}dV_{\sigma_{\infty}},\label{eq-VolumeBoundedByL^n}
        \\\Area\left(\partial M,\frac{\sigma_j}{h_j}\right)&\le  \int_M\left|\frac{\sigma_j}{h_j}\right|_{\sigma_{\infty}}^{\frac{n-1}{2}}dV_{\sigma_{\infty}},
    \end{align}
    and so one can again use H\"{o}lder's inequality to find sufficient conditions on $h_j$ and $\sigma_j$ separately which imply \eqref{eq-AreaBound}. To imply \eqref{eq-VolumeBound}, one can use Lemma 4.3 of \cite{Allen-Sormani-2} combined with \eqref{eq-VolumeBoundedByL^n} and H\"{o}lder's inequality.
\end{rmrk}

\begin{rmrk}
Due to the fact that the null distance depends on a choice of time function, a natural question is what should be taken as a canonical choice for the time function when studying a sequence of spacetimes? One way to approach this question is to ask what properties you would like the null distance function to have and investigate what kind of time functions will yield the desired properties for the null distance. This investigation was initiated by A. Sakovich and C. Sormani \cite{SS}, and A. Burtscher and L. Garc\'{i}a-Heveling \cite{BG2} who both give independent and slighlty different characterizations of what type of time functions yield null distances which encode causality. Encoding causality is a desirable property because it means that the metric space structure is compatible with the causal structure of the spacetime. The class of time functions which (locally) encode causality is still a fairly large class so we can still ask for a good choice of a canonical time function.

     One suggestion is to use the cosmological time function, defined by L. Anderson, G. Galloway, and R. Howard \cite{AGH}, as a good choice for a canonical time function when defining the null distance on a general spacetime (See also the work of M. Hoseini, N. Ebrahimi, and M. Vatandoost \cite{MNM} for a recent study of the cosmological time function on Lorentzian length spaces). We note that the time coordinate $t$ for the static spacetimes investigated in this paper is not the cosmological time function  for these spacetimes. That being said, for the class of static spacetimes in this paper (and also for globally hyperbolic spacetimes), the $t$ coordinate is a natural candidate for a time function and allows for better estimates of the null distance. In particular, for static spacetimes as defined in this paper, the $t$ coordinate locally encodes causality by Theorem 1.1 of \cite{SS}. It is also a natural canonical choice given the way that the static spacetimes are expressed as metrics and so we are justified in choosing $t$ as the canonical time function when studying sequences of static spacetimes. 
\end{rmrk}

We now prove a similar Theorem which provides conditions which imply Sormani-Wenger intrinsic flat convergence of a sequence of static spacetimes equipped with the null distance metric space. We point out that Example \ref{Ex-BlowUpSpline} shows that the relaxed conditions of Theorem \ref{thm-MainTheorem SWIF} are not enough to imply uniform or Gromov-Hausdorff convergence and hence Sormani-Wenger Intrinsic Flat convergence is the correct conclusion for this theorem.

 \begin{thm}\label{thm-MainTheorem SWIF}
Let $0<A<\infty$, $M$ compact, connected, and oriented manifold, $(M,\sigma_j)$ be a sequence of continuous Riemannian manifolds, $(M,\sigma_{\infty})$ a smooth Riemannian manifold,  $L=[t_0,t_1]\times M$, $h_j:M \rightarrow (0,\infty)$, $g_j=-h_j^2dt^2+\sigma_j$, $j \in \N\cup\{\infty\}$, with $h_j$ continuous for $j \in \N$ and smooth for $j =\infty$. If 
\begin{align}
 \frac{\sigma_j(v,v)}{h_j^2} &\ge  \left( 1 -\frac{1}{j} \right)\frac{\sigma_{\infty}(v,v)}{h_{\infty}^2} \quad \forall x \in M, v \in T_xM, \label{eq-Space Component lower bound}
\\\Vol\left(M,\frac{\sigma_j}{h_j}\right)&\rightarrow \Vol\left(M,\frac{\sigma_{\infty}}{h_{\infty}}\right),
\\ \Area\left(\partial M,\frac{\sigma_j}{h_j}\right) &\le A,
\end{align}
 then
\begin{align}
    (L,\hat{d}_{t,g_j})\rightarrow (L,\hat{d}_{t,g_{\infty}}),
\end{align}
in the Sormani-Wenger intrinsic flat sense. 
\end{thm}

\begin{rmrk}
    One should note that both Theorem \ref{thm-MainTheorem GH} and Theorem \ref{thm-MainTheorem SWIF} can be applied to non-compact manifolds with boundary where $h_j \rightarrow \infty$ as one approaches the boundary or $\sigma_j$ becomes degenerate. The theorems are applied by taking a precompact exhaustion of the non-compact manifold with boundary and considering convergence on compact subsets. So the results of this paper apply to a sequence of Schwarzschild manifolds with vanishing mass as well as similar sequences of static manifolds which one does not expect to converge smoothly on compact subsets.
\end{rmrk}

\begin{rmrk}
    If $h_j=1$ and $\sigma_j$ is given as a graph of a function $f_j:M\rightarrow (0,\infty)$ so that $\sigma_j=\sigma_{\infty}+df_j^2$ then we see that \eqref{eq-Space Component lower bound} is automatically satisfied. This shows that the stability results of L.-H. Huang, D. Lee and C. Sormani \cite{HLS}, L.-H. Huang, D. Lee and R. Perales \cite{HLP}, and A. J. Cabrera Pacheco, M. Graf and R. Perales \cite{PGPHyp}   in the time symmetric case can be extended to stability results for the associated static manifolds equipped with the null distance by Theorem \ref{thm-MainTheorem SWIF}. This is because the volume convergence and area bound hypotheses were also obtained in the previous works on the graphical case of stability of the positive mass theorem.
\end{rmrk}

Bernal and S\'{a}nchez \cite{Bernal-Sanchez} show that a globally hyperbolic spacetime $(M,g)$ can be written in the form $g=-h^2dt^2+\hat{g}(t)$ where $M=\R\times M$ for some manifold $\Sigma$, $h: \R\times M\rightarrow (0,\infty)$, and $\hat{g}(t)$ is a Riemannian metric on $M$ for each $t \in \R$. So it is natural to ask for a generalization of Theorem \ref{thm-MainTheorem GH} and Theorem \ref{thm-MainTheorem SWIF} for sequences of globally hyperbolic spacetimes. In a similar way as in the warped product and static cases, since the null distance is conformally invariant, i.e $\hat{d}_{\tau, \lambda^2 g} = \hat{d}_{\tau, g}$ where $\lambda: \R \times M \rightarrow (0,\infty)$, we can equivalently study $h^{-2}g=-dt^2+h^{-2}\hat{g}(t)$ with respect to the null distance, as we will see in the following conjecture.

\begin{conj}\label{conj-Conjecture SWIF}
Let $0<A<\infty$ and $M$ be a compact, connected, and oriented  manifold, $L=[t_0,t_1]\times M$, $h_j: [t_0,t_1]\times M \rightarrow (0,\infty)$, $g_j=-h_j^2dt^2+\sigma_j(t)$, $j \in \N\cup\{\infty\}$, $\sigma_j(t)$ is a Reimannian manifold for each $t \in [t_0,t_1]$, continuous for $j \in \N$ and smooth when $j=\infty$. If 
\begin{align}
 \frac{\sigma_j(v,v)}{h_j^2} &\ge  \left( 1 -\frac{1}{j} \right)\frac{\sigma_{\infty}(v,v)}{h_{\infty}^2} \quad \forall x \in M, v \in T_xM, \label{eq-Space Component lower bound}
\\\int_{t_0}^{t_1}\Vol\left(M,\frac{\sigma_j(t)}{h_j(t)}\right)dt&\rightarrow \int_{t_0}^{t_1}\Vol\left(M,\frac{\sigma_{\infty}(t)}{h_{\infty}(t)}\right)dt,
\\ \max_{t \in [t_0,t_1]}\Area\left(\partial M,\frac{\sigma_j(t)}{h_j(t)}\right) &\le A,
\end{align}
 then
\begin{align}
    (L,\hat{d}_{t,g_j})\rightarrow (L,\hat{d}_{t,g_{\infty}}),
\end{align}
in the Sormani-Wenger intrinsic flat sense. 
\end{conj}

One should note that the tools of section \ref{sect-Estimating SWIF} still apply to Conjecture \ref{conj-Conjecture SWIF} but the arguments used in section \ref{sect-Proofs} are not available in this case. Hence a new way to obtain pointwise convergence of the distance function $\hat{d}_{t,g_j}$ almost everywhere needs to be developed in order to take advantage of the tools of section \ref{sect-Estimating SWIF}. This will not be a straightforward adaption of the work of \cite{Allen-Perales-Sormani} or of this paper since the Lorentzian structure and the definition of null distance will have to be taken head on with a new insight in this case. That being said, the examples explored and main theorems in \cite{AB,Allen-Null}, as well as this paper, provide strong evidence for Conjecture \ref{conj-Conjecture SWIF}.

In section \ref{sect-Background}, we review the definition of the null distance, properties which have been established for this distance, the definition of SWIF convergence, as well as the work in \cite{Allen-Perales-Sormani}, \cite{Allen-Perales}, \cite{Allen-Holder}, and \cite{Allen-Bryden} which will be needed to prove the main theorems of this paper.

In section \ref{sect-Examples}, we review examples given in \cite{AB} and \cite{Allen-Null} which shed light on the necessity of the various hypotheses in the main theorems.

In section \ref{sect-Estimating SWIF}, we adapt a technique for estimating the SWIF distance, first developed for seqeunces of Riemannian manifolds in \cite{Allen-Perales-Sormani}, to the setting of globally hyerpbolic spacetimes equipped with the null distance.

In section \ref{sect-Proofs}, we combine the tools reviewed in section \ref{sect-Background} with estimation techniques of section \ref{sect-Estimating SWIF} to prove the main theorems of this paper.

\section{Background}\label{sect-Background}

A spacetime is a time oriented Lorentzian manifold $(L,g)$. We define $J^{\pm}_g(p)$ to be the set of all points which are in the causal future or past of $p$, with respect to $g$. Let $\tau:L \rightarrow \R$ be a time function which is a continuous function that is strictly increasing along all future directed causal curves. We now define the null length and null distance introduced by Sormani and Vega \cite{SV}. Let $\beta :[a,b] \rightarrow L$ be a piecewise causal curve, i.e. a piecewise smooth curve that is either future-directed or past-directed causal on its pieces $a=s_0 < s_1 < \ldots < s_k =b$ (see figure \ref{fig:AdmissableCurves}). The null length of $\beta$ is given by
 \begin{align}
  \hat L_{\tau,g} (\beta) = \sum_{i=1}^k |\tau(\beta(s_i))-\tau(\beta(s_{i-1}))|.
 \end{align}
% Since $\tau$ is strictly increasing (decreasing) along any future (past) directed causal curve, for any causal curve $\gamma(s)$ we find that $\tau(\gamma(s))$ is differentiable almost everywhere with respect to $s$.  Since a piecewise causal curve has only finitely many breaks we can compute the null length of any piecewise causal curve $\beta$ by
%  \textcolor{red}{It may not be absolutely continuous and hence maybe not integrable.}

In the case where $\tau$ is differentiable we can compute the length of any piecewise causal curve $\beta$ by 
 \begin{align}
  \hat L_{\tau,g} (\beta) = \int_a^b |(\tau \circ \beta)'(s)| ds.
 \end{align}
We will take advantage of this integral formula when estimating the SWIF distance between static spacetimes in section \ref{sect-Estimating SWIF}. For any $p,q \in M$, the \emph{null distance} is given by
 \begin{align}
  \hat d_{\tau,g} (p,q) = \inf \{ \hat L_{\tau,g} (\beta) : \beta \text{ is a piecewise causal curve from } p \text{ to } q \}.  
 \end{align}
 % \begin{figure} [h]
 %  \centering
 % \begin{tikzpicture}[scale=1]
 %  \draw[<->] (-1.5,0) -- (1.5,0) node[anchor=north west]{} ;
 %  \draw[->] (0,0) -- (0,1.2) node[anchor=south east]{};
 %  \draw[->] (2,0) -- (3.5,0) node[anchor=north west]{};
 %  \draw[->] (2,0) -- (2,1.2) node[anchor=south east]{};
 %    \draw[<->] (4,0) -- (7,0) node[anchor=north west]{};
 %  \draw[->] (5.5,0) -- (5.5,1.2) node[anchor=south east]{};
 %  \draw[fill] (-1,1) circle [radius=0.05];
 %  \node[left, outer sep=2pt] at (-1,1) {p};
 %  \draw[fill] (1,1) circle [radius=0.05];
 %  \node[right, outer sep=2pt] at (1,1) {q};
 %   \draw[ thick] (-1,1) -- (-1/2,1/2) -- (0,1) -- (1/2,1/2) -- (1,1);
 %   \draw[thick](2+.25,.25)--(2+1,1)-- (2+1.25,.75) ;
 %    \draw[fill] (2+.25,.25) circle [radius=0.05];
 %   \node[right, outer sep=2pt] at (2+.25,.25) {p};
 %  \draw[fill] (2+1.25,.75) circle [radius=0.05];
 %  \node[right, outer sep=2pt] at (2+1.25,.75) {q};
 %  \draw[thick](4.5,.25)--(4.5,1)-- (5.5+-3/4,3/4) -- (5.5+-1/2,1) -- (5.5+-1/4,3/4) -- (5.5+0,1)--(5.5+1/4,3/4) -- (5.5+1/2,1) -- (5.5+3/4,3/4)--(6.5,1)--(6.5,1.2); 
 %    \draw[fill] (4.5,.25) circle [radius=0.05];
 %   \node[left, outer sep=2pt] at (4.5,.25) {p};
 %  \draw[fill] (6.5,1.2) circle [radius=0.05];
 %  \node[right, outer sep=2pt] at (6.5,1.2) {q};
 % \end{tikzpicture}
 %  \caption{Several examples of admissible curves in Minkowski space.}
 %  \label{fig:AdmissableCurves}
 % \end{figure}
 
It is important to note that the null distance is compatible with weak notions of spacetimes such as the Lorentzian length spaces of  M. Kunzinger and C. S\"{a}mann \cite{KCS}. For further properties of null distance see the work by the author and A. Burtscher \cite{AB}, A. Burtscher and L. Garc\'{i}a-Heveling \cite{BG,BG2}, G. Galloway \cite{GG}, M. Graff and C. Sormani \cite{GS}, M. Kunzinger and R. Steinbauer \cite{KS}, B. Meco, A. Sakovich and C. Sormani \cite{MSS}, A. Sakovich and C. Sormani \cite{SS,SS2}, C. Sormani and C. Vega \cite{SV}, and C. Vega \cite{V}.
One particularly important property which we will exploit several times in this paper is the conformal invariance of the null distance which was noticed in \cite{SV} and follows from the fact that the set of causal curves on a spacetime is a conformal invariant.

 \subsection{Pointwise Convergence on Manifolds with Boundary}

 In order to prove Theorem \ref{thm-MainTheorem GH} we will first establish pointwise convergence of the null distance functions. This will be done by first establishing pointwise convergence of the distance functions with respect to the Riemannian manifold which is part of the static spacetime structure. To this end we will apply the following theorem of the author and E. Bryden \cite{AB} which extends the work in \cite{Allen-Perales-Sormani, Allen-Perales} where similar theorems were established under less general hypotheses.

 \begin{thm}[Theorem 4.5 of \cite{Allen-Bryden}]\label{thm-Allen-Bryden-VADB-II}
  Let $(M,\partial{}M,g_{0})$ be a Riemannian manifold with boundary, and let $g_{i}$ be 
  a sequence of Riemannian metrics so that
  \begin{align}
\Diam(M,g_j) &\le D,
  \\  \Vol(M,g_j) &\rightarrow \Vol(M,g_0)
  \\ \Area(\partial M,g_j) &\le A,
  \\ g_j(v,v) &\ge ( 1 - C_j)g_0(v,v), \quad \forall p \in M, v \in T_pM, \quad C_j \searrow 0,
\end{align}
  Then $d_{g_{j}}$ converges to $d_{g_{0}}$ pointwsie almost everywhere on $M\times{}M$ with respect to
  $dV_{g_{0}}\otimes{}dV_{g_{0}}$.
\end{thm}

\subsection{$L^p$ bounds on Riemannian Manifolds and H\"{o}lder Distance Bounds}

After we establish pointwise convergence of the sequence of null distance functions to the desired limiting null distance function, we will apply a compactness theorem to conclude uniform convergence. Since uniform convergence implies Gromov-Hausdorff convergence we will then obtain Theorem \ref{thm-MainTheorem GH}. The key ingredient in this part of the proof is a H\"{o}lder bound for Riemannian metrics, first obtained by the author in \cite{Allen-Holder}, which we review for the reader below. Since we will need this result for compact manifolds with boundary we also give a simple proof which extends the original result to manifolds with boundary.

 \begin{thm}[Theorem 1.1 of \cite{Allen-Holder}]\label{Thm-Allen-Holder}
Let $M^n$ be a compact,  connected, and oriented manifold, $M_0=(M,\sigma_0)$ a smooth Riemannian manifold, and $M_1=(M,\sigma_1)$ a continuous Riemannian manifold. If
 \begin{align}
 \exists p > n, \quad \|g_1\|_{L_{g_0}^{\frac{p}{2}}(M)} \le C
 \end{align}
  then 
\begin{align}
d_{\sigma_1}(q_1,q_2) \le C'(M,\sigma_0) d_{\sigma_0}(q_1,q_2)^{\frac{p-n}{p}},\quad \forall q_1,q_2 \in M.
\end{align}
\end{thm}
\begin{proof}
    If $M$ is closed then we are done. If $M$ has boundary, then let $N$ be the double of $N$ which induces a natural continuous Riemannian metrics $N_0=(N,\sigma_0)$, and $N_1=(N,\sigma_1)$. Since we need $N_0$ to be smooth, we can run Ricci flow for a short period of time by Theorem 1.1 of P. Burkhardt-Guim \cite{PBG} to obtain a smooth Riemannian metric $\hat{\sigma}_0$ which is $C^0$ close to $\sigma_0$. Then since the corresponding distance functions will satisfy a H\"{o}lder bound
   \begin{align}\label{eq-Intermediate Holder Bound}
d_{\hat{\sigma}_0}(q_1,q_2) \le\bar{C}(M,\sigma_0) d_{\sigma_0}(q_1,q_2)^{\frac{p-n}{p}},\quad \forall q_1,q_2 \in M,
\end{align}
we may apply the result of \cite{Allen-Holder} to $N_1$ and $(N,\hat{\sigma}_0)$ and combine with \eqref{eq-Intermediate Holder Bound} to obtain the desired result on the original manifolds with boundary.
\end{proof}

\subsection{Uniform Distance Bounds on a Good Set}\label{sect-Pointwise to Unfirom on a Good Set}

In this section we remind the reader how to obtain distance controls on a closed and measurable good set $W$ of almost full volume. This was first established by the author, R. Perales, and C. Sormani \cite{Allen-Perales-Sormani} and was used to establish the Volume Above Distance Below theorem, which is the inspiration for the main theorems of this paper as well as Conjecture \ref{conj-Conjecture SWIF}. By combining Lemma 4.3 of \cite{Allen-Perales} with Theorem \ref{thm-Allen-Bryden-VADB-II} of \cite{Allen-Bryden} one obtains the following theorem which will be used to prove Theorem \ref{thm-MainTheorem SWIF}.

\begin{thm}\label{thm-Uniform Distance Bounds on W}
Let $M$ be a compact, connected, and oriented manifold, $(M,\sigma_j)$ be a sequence of continuous Riemannian manifolds and $(M,\sigma_0)$ a smooth Riemannian manifold such that
\begin{align} 
\sigma_0(v,v) \le \sigma_j(v,v), \quad \forall p \in M, v \in T_pM,
\end{align}
\begin{align}
\diam(M,\sigma_j) \le D_0,
\end{align}
and 
\begin{align} 
\Vol_j(M_j) \rightarrow \Vol_0(M_0).
\end{align}  
Then for any $\lambda  \in (0, \diam(M_0))$ and  $\kappa >1$, there exists 
 a measurable set $W_{\lambda,\kappa} \subset M$  and a constant $\delta_{\lambda, \kappa, j}>0$ such that for all $p_1, p_2 \in W_{\lambda, \kappa}$
\begin{align}
|d_{\sigma_j}(p_1,p_2)-d_{\sigma_0}(p_1,p_2)| < 2 \lambda + 2\delta_{\lambda,\kappa,j},
\end{align}
where $\delta_{\lambda,\kappa,j} \to 0$ as $j \to \infty$,
and
\begin{align}
\Vol(M \setminus W_{\lambda, \kappa},\sigma_j) \le  \frac{1}{\kappa}\Vol(M,\sigma_0)+|\Vol(M,\sigma_j)-\Vol(M,\sigma_0)|.
\end{align}
\end{thm}

\subsection{Uniform Convergence of Metric Spaces}
Given two metric spaces $(X,d_1)$ and $(X,d_2)$, defined on the same set $X$, we can define the uniform distance between them to be
\begin{align}\label{def-UniformDist}
    d_{unif}(d_0,d_1)=\sup_{x_1,x_2\in X}|d_1(x_1,x_2)-d_2(x_1,x_2)|
\end{align}
Given a sequence of metric spaces $(X,d_j)$ we can define the uniform convergence of $d_j$ to a limiting metric space $(X,d_{\infty})$ by
\begin{align}\label{def-UniformConvergence}
    d_{unif}(d_j,d_{\infty}) \rightarrow 0.
\end{align}

\subsection{Sormani-Wenger Intrinsic Flat Convergence}

The author and A. Burtscher \cite{AB} showed that warped product spacetimes, and more generally globally hyperbolic spacetimes, are integral current spaces (see the definition below) using bi-Lipschitz estimates. Since these are the type of metric spaces which Sormani-Wenger Intrinsic Flat (SWIF) distance is defined on, we are justified in applying the SWIF notion of distance to static spacetimes equipped with the null distance. In this subsection we will remind the reader of some important definitions related to the SWIF distance which will be crucial for estimates we make in section \ref{sect-Estimating SWIF}.

We start by reviewing the definition of the flat distance of H. Federer and W. H. Fleming \cite{FF} which was extended to arbitrary metric spaces by L. Ambrosio and B. Kircheim \cite{AK}. Let $(Z,d)$ be a complete metric space, $\lip(Z)$ the set of  real valued Lipschitz functions on $Z$, and 
$\lip_b(Z)$ the bounded ones.  An $n$-dimensional current $T$ on $Z$ 
is a multilinear map $T: \lip_b(Z) \times [\lip(Z)]^n  \to \R$  that satisfies properties which can be found in Definition 3.1 of  \cite{AK}.  
From the definition of $T$ we know there exists a finite Borel measure on $Z$,  $\|T\|$,  called the mass measure of $T$.
Then the mass of $T$ is defined as $\mass(T)=\|T\|(Z)$.  The  boundary of $T$,  $\partial T: \lip_b(Z) \times [\lip(Z)]^{n-1}  \to \R$  
is the linear functional given by
\begin{align}
\partial T(f, \pi) = T(1, (f, \pi)),
\end{align}
and for any Lipschitz function $\varphi: Z \to Y$  the push forward of $T$,   ${\varphi}_{\sharp} T : \lip_b(Y) \times [\lip(Y)]^{m}  \to \R$
is the $n$-dimensional current given by 
\begin{align}
{\varphi}_{\sharp} T (f, \pi) 
= T( f\circ \varphi, \pi \circ \varphi ).
\end{align}
Furthermore,  the following inequality holds
\begin{equation}\label{eq-pushMeasure}
\|  \varphi_\sharp T\|  \leq \lip(\varphi)^n \varphi_\sharp \|T\|. 
\end{equation} 

More generally, an $n$ dimensional integer rectifiable current, $T$, can be parametrized by
a countable collection of biLipschitz charts, $\varphi_i: A_i \to \varphi_i(A_i)\subset Z$ where 
$A_i$ are Borel in ${\mathbb R}^n$ with pairwise disjoint images and integer weights $\theta_i\in {\mathbb Z}$ such that
\begin{align}
T(f, \pi_1,...\pi_n) = \sum_{i=1}^\infty \theta_i \int_{A_i} (f\circ \varphi_i)\, d(\pi_1 \circ \varphi_i)\wedge \cdots \wedge d(\pi_n \circ \varphi_i)
\end{align}
has finite mass, $\mass(T)=\|T\|(Z)$.
 
An $n$-dimensional integral current in $Z$ is an $n$-dimensional current that can be written as a countable sum of terms, 
\begin{align}
    T=  \sum_{i=1}^\infty \varphi_{i\sharp} [[\theta_i]],
\end{align} 
with $\theta_i \in L^1(A_i, \mathbb R)$ integer constant functions, 
such that $\partial T$ is a current.  The class that contains all $n$-dimensional integral currents of $Z$ is denoted by $\intcurr_n(Z)$. 
For  $T\in I_n(Z)$,  L. Ambrosio and B. Kirchheim proved that the subset 
\begin{align}
\set(T)= \left\{ z \in Z \, | \, \liminf_{r \downarrow 0} \frac{\|T\|(B_r(z))}{ r^n }> 0 \right\}
\end{align}   
is $\mathcal H^n$-countably recitifiable. That is, $\set(T)$ can be covered by images of Lipschitz maps from $\R^n$  to $Z$
up to a set of zero  $\mathcal H^n$-measure.

Since all the examples of this paper will be static spacestimes, or more generally globally hyerbolic spacetimes, we now look at the example of globally hyerbolic spacetimes as integral current spaces as observed in \cite{AB}.

\begin{ex}\label{ex-Globally Hyperbolic Integral Current Space}
For an  $n+1$-dimensional compact globally hyperbolic spacetime $(L^{n+1},\bar{g}=-h^2dt^2+g)$,
the triple $(L,\hat{d}_{t,\bar{g}}, [[L]])$ given as follows is an $n+1$-dimensional integral current space.
Here $[[L]] :  \lip_b(L) \times [\lip(L)]^{n+1} \to \R$ is given by 
\begin{align}\label{eq-canonicalT}
[[L]] = &  \sum_{i,k} {\psi_i}_\sharp [[1_{A_{ik}}]] 
\end{align}
where we have chosen the smooth locally finite atlas $\{(U_i, \psi_i)\}_{i \in \mathbb N}$ of $L$ given in \cite{AB} consisting of positively oriented Lipschitz charts, 
$\psi_i :  U_i  \subset \R^n   \to L$  and $A_{ik}$  are precompact Borel sets such that $\psi_i(A_{ik})$ have disjoint images for all $i$ and $k$, and cover $L$ $\mathcal H^n$-almost everywhere. In this case, $\|[[L]]\|= dV_{\hat{g}}$ where $\hat{g}=h^2dt^2+g$. 
\end{ex}

In \cite{AK}, Ambrosio-Kirchheim prove that for rectifiable currents
\begin{align}
\|T\|=  \lambda \theta \mathcal{H}^n
\end{align}
where $\theta$ is an integer valued function and the area factor
$\lambda:  \set(T)  \to \R$ is a measurable function bounded above by 
\begin{align} \label{C_n}
C_n=2^n/\omega_n \textrm{ where } \omega_n=\Vol_{{\mathbb E}^n}(B_0(1)).
\end{align}
So that
\begin{align}
\mass(T) \le C_n  \sum_{i=1}^\infty |\theta_i| \mathcal{H}^n( \varphi_i(A_i)) < \infty.
\end{align}

We will make frequent use of this in section \ref{sect-Estimating SWIF} in order to trade mass estimates of integral current spaces for Hausdorff measure estimates of metric spaces. In particular, we will use $\mathcal{H}^n_d$ to denote the Hausdorff measure with respect to the metric $d$ and we will use the estimate, which follows from the observations above, that
\begin{align}
    \|T\|_d(Z) \le C_n \mathcal{H}^n_d(Z),
\end{align}
where $(Z,d,T)$ has multiplicity $1$.

The flat distance between two integral currents $T_1, T_2  \in  \intcurr_{n} (Z)$ is defined as
\begin{align}
\begin{split}
d_{F}^Z( T_1, T_2)=\inf\Bigl\{  \mass(U)+ \mass(V)\,  |&
\, \, U \in \intcurr_{n}(Z), \, V   \in  \intcurr_{n+1} (Z),
\\
 & \, \, T_2 -T_1 =U + \partial V  \Bigr\}.    
\end{split}
\end{align}

With the definition of flat convergence on a general metric space in hand we are ready to define integral current spaces which are the spaces for which Sormani-Wenger intrinsic flat distance is defined. One should see C. Sormani and S. Wenger \cite{Sormani-Wenger} for more details. An $n$-dimensional integral current space $(X, d, T)$ consists of a metric space $(X, d)$ and an $n$-dimensional integral current defined on the completion of $X$, $T\in I_n(\bar{X})$, such that $\set(T)=X$.

We say that an integral current space $(X,d,T)$ is precompact if $X$ is precompact with respect to $d$.
Given two $n$-dimensional integral current spaces, $(X_1, d_1, T_1)$ and $(X_2, d_2, T_2)$, a current preserving isometry between them is a 
 metric isometry $\varphi: X_1 \to X_2$  such that $\varphi_\sharp T_1=T_2$.   We are now ready to state the definition of the SWIF distance between integral current spaces.

\begin{defn}[Sormani-Wenger \cite{Sormani-Wenger}]\label{defn-SWIF} 
Given two $n$-dimensional precompact integral current spaces 
$(X_1, d_1, T_1)$ and $(X_2, d_2,T_2)$, the Sormani-Wenger Intrinsic Flat distance   
between them is defined as
\begin{align}
\begin{split}
d_{\mathcal{F}}&\left( (X_1, d_1, T_1), (X_2, d_2, T_2)\right)
\\&=\inf  \Bigl\{d_F^Z(\varphi_{1\sharp}T_1, \varphi_{2\sharp}T_2)|  \,
 (Z,d_Z) \text{ complete} ,\, \varphi_j: X_j \to Z   \text{ isometries}\Bigr\}.      
\end{split}
\end{align}
\end{defn}

The function $d_{\mathcal{F}}$ is a distance up to current preserving isometries. We say that a sequence $(X_j,d_j,T_j)$ of $n$-dimensional integral current spaces converges in the volume preserving Sormani-Wenger Intrinsic Flat sense, $\mathcal{VF}$, to $(X,d,T)$ if the sequence converges with respect to the intrinsic flat distance
to $(X,d,T)$ and the masses $\mass(T_j)$ converge to $\mass(T)$.  

\section{Examples}\label{sect-Examples}

In this section we review examples which were explored in \cite{AB, Allen-Null} which demonstrate the importance of the hypotheses in Theorem \ref{thm-MainTheorem GH} and Theorem \ref{thm-MainTheorem SWIF} and build intuition for uniform convergence of the null distance function in the static case. Throughout this section $(\mathbb{D}^n,\sigma)$ stands for the closed flat unit disk which is chosen for notational convenience. It should be noted that the conclusions of the examples does not rely heavily on this choice and similar examples will hold for a compact, connected Riemannian manifold $(\Sigma^n,\sigma)$ replacing $(\mathbb{D}^n, \sigma)$.

\subsection{No uniform  control from below in the space direction}

The first example demonstrates the necessity of the assumption on $\sigma_j \ge \sigma_{\infty}$ in the main theorems. We will see that without this assumption one can have a sequence of static metrics which converge to a metric space which is not a spacetime.

\begin{ex}\label{Ex-No Control From Below in Space}
      Let $(\mathbb{D}^n,\sigma)$, $n \ge 2$ be a flat disk, $f_j:\mathbb{D}^n \rightarrow (0,\infty)$, $j \ge 2$ a sequence of continuous functions defined radially on $\mathbb{D}^n$ where $s$ is the distance in $(\mathbb{D}^n,\sigma)$ to the boundary by
      \begin{align}f_j(s)=
          \begin{cases}
          \frac{1}{j} & s\in  \left[0,\frac{1}{j}\right]
          \\k_j(s)& s \in   \left[\frac{1}{j},\frac{3}{2j}\right]
\\1& \text{ otherwise }
          \end{cases}
      \end{align}
      where $k_j$ is any increasing continuous function so that $k_j(\frac{1}{j})=\frac{1}{j}$, and $k_j(\frac{3}{2j})=1$. If $g_j=-dt^2+f_j^2\sigma$ and $g_0=-dt^2+\sigma$ then we find $\hat{d}_{t,g_j} \not\rightarrow \hat{d}_{t,g_0}$. Furthermore, we let $L=([0,1]\times \mathbb{D}^n,\hat{d}_{t,g_0})$ and let
      \begin{align}
         F: [0,1]\times \partial \mathbb{D}^n \subset L \rightarrow  [0,1]
      \end{align} 
      be the map defined by projection onto to the time factor of $L$. Then if $(P,d_0)=(L, \hat{d}_{t,g_0}/\sim)$ where we identify points by the map $F$ we can conclude that   
      \begin{align}
          ([0,1]\times \mathbb{D}^n,\hat{d}_{t,g_j}) \rightarrow (P,d_0),
      \end{align}
      in the Gromov-Hausdorff sense.
  \end{ex}

\subsection{No uniform control from above in the time direction}

The next example demonstrates the necessity of assuming $h_{\infty} \ge h_j$ since we will see that without this assumption we can have a sequences of static spacetimes converge to a metric space which is not a spacetime. One should note that the next example is equivalent to the previous example as far as null distance is concerned since the null distance is conformally invariant. When we combine these two examples we see the necessity of assuming $\frac{\sigma_j}{h_j} \ge \frac{\sigma_{\infty}}{h_{\infty}}$ in the main theorems.

\begin{ex}\label{Ex-No Control From Above In Time}
      Let $(\mathbb{D}^n,\sigma)$, $n \ge 2$ be a flat disk, $h_j:[0,1]\rightarrow (0,\infty)$, $j \ge 2$ a sequence of continuous functions defined radially on $\mathbb{D}^n$ where $s$ is the distance in $(\mathbb{D}^n,\sigma)$ to the boundary by
      \begin{align}h_j(s)=
          \begin{cases}
          j & s\in  \left[0,\frac{1}{j}\right]
          \\\bar{k}_j(s)& s \in   \left[\frac{1}{j},\frac{3}{2j}\right]
\\1& \text{ otherwise }
          \end{cases}
      \end{align}
      where $\bar{k}_j$ is any decreasing continuous function so that $\bar{k}_j(\frac{1}{j})=j$ and $\bar{k}_j(\frac{3}{2j})=1$. If $g_j=-h_j^2dt^2+\sigma$ and $g_0=-dt^2+\sigma$ then we find $\hat{d}_{t,g_j} \not\rightarrow \hat{d}_{t,g_0}$. Furthermore, we let $L=([0,1]\times \mathbb{D}^n,\hat{d}_{t,g_0})$ and let
      \begin{align}
         F: [0,1]\times \partial \mathbb{D}^n \subset L \rightarrow  [0,1]
      \end{align} 
      be the map defined by projection onto to the time factor of $L$.  Then if $(P,d_0)=(L, \hat{d}_{t,g_0}/\sim)$ where we identify points by the map $F$ in order to conclude that    
      \begin{align}
          ([0,1]\times \mathbb{D}^n,\hat{d}_{t,g_j}) \rightarrow (P,d_0),
      \end{align}
      in the Gromov-Hausdorff sense.
  \end{ex}

  \begin{proof}
     By the conformal invariance of the null distance, if we define $f_j = \frac{1}{h_j}$ then $f_j^2g_j = -dt^2+f_j^2\sigma$ is the metric of Example \ref{Ex-No Control From Below in Space}. Hence the uniform convergence of this example follows from the uniform convergence of Example \ref{Ex-No Control From Above In Time}, which is proved in \cite{Allen-Null}. 
  \end{proof}

\subsection{Volume bound without volume convergence produces bubbles}
 In the next two examples we will see that the rate of blow up of the sequence is crucial to understanding the geometry of the limit. In particular, in the next example we see that if we have a volume bound but no volume convergence then we should not expect the main theorems of this paper to hold.

  \begin{ex}\label{Ex-BlowUpBubble}
      Let $(\mathbb{D}^n,\sigma)$, $n \ge 2$ be a flat disk, $f_j: \mathbb{D}^n \rightarrow (0,\infty)$, $j \ge 2$ a sequence of continuous functions defined radially on $\mathbb{D}^n$ by
      \begin{align}f_j(r)
          \begin{cases}
          j & r\in  \left[0,\frac{1}{j}\right]
          \\h_j(r)& r \in   \left[\frac{1}{j},\frac{3}{2j}\right]
\\1& \text{ otherwise }
          \end{cases}
      \end{align}
      where $h_j$ is any decreasing continuous function so that $h_j(\frac{1}{j})=j$, $h_j(\frac{3}{2j})=1$, and $\displaystyle \int_{\frac{1}{j}}^{\frac{3}{2j}} h_jdr\le \frac{C}{j}$. If $g_j=-dt^2+f_j^2\sigma$ and $g_0=-dt^2+\sigma$ then we find $\hat{d}_{t,g_j} \not\rightarrow \hat{d}_{t,g_0}$. Furthermore, we let  $N_1=([0,1]\times \mathbb{D}^n,\hat{d}_{t,g_0})$, $N_2=([0,1]\times \mathbb{D}^n,\hat{d}_{t,g_0})$, and 
      \begin{align}
         F:[0,1] \times \partial \mathbb{D}^2\subset N_1 \rightarrow  [0,1]\times\{0\}\subset N_2
      \end{align} 
      defined by projection of the second factor.  Then if $(N_1\sqcup N_2, \hat{d}_{t,g_0})$ is the disjoint union of $N_1$ and $N_2$ with the null distance then we can define the metric space $(P,d_0)=(N_1 \sqcup N_2, \hat{d}_{t,g_0}/\sim)$ where we identify points by the map $F$ in order to conclude that    
      \begin{align}
          ([0,1]\times \mathbb{D}^n,\hat{d}_{t,g_j}) \rightarrow (P,d_0),
      \end{align}
      in the Gromov-Hausdorff sense.
  \end{ex}
 
\subsection{Volume convergence without $L^{\frac{p}{2}}$, $p >n$ bound separates GH from SWIF}

In the next example we see that for a sequence of functions blowing up at a critical rate the limit will be Minkowski space with a taxi metric defined on $[0,1]\times [0,1]$ attached to the $t$-axis. This example illustrates the difference between Theorem \ref{thm-MainTheorem GH} and Theorem \ref{thm-MainTheorem SWIF}. Without the $L^{\frac{p}{2}}$, $p > n$ bound, but with volume convergence (which is equivalent to $L^n$ convergence of $f_j$ below), we see that the GH limit will consist of a metric space which is not a spacetime but the SWIF limit will consist of a spacetime.

   \begin{ex}\label{Ex-BlowUpSpline}
      Let $(\mathbb{D}^n,\sigma)$, $n \ge 2$ be a flat disk, $f_j:[0,1]\times \mathbb{D}^n$ a sequence of continuous functions defined radially on $\mathbb{D}^n$ by
      \begin{align}f_j(r)
          \begin{cases}
          \frac{j^{\lambda}}{1+\lambda \ln(j)} & r \in \left[0,\frac{1}{j^{\lambda}}\right]
         \\ \frac{1}{r(1-\ln(r))} & r \in \left[\frac{1}{j^{\lambda}},\frac{1}{j}\right]
          \\h_j(r) & r \in  \left[\frac{1}{j},\frac{3}{2j}\right]
\\1& \text{ otherwise }
          \end{cases}
      \end{align}
      where $\lambda >1$ and $h_j$ is any decreasing, continuous function so that $h_j(\frac{1}{j})= \frac{j}{1+\ln(j)}$, $h_j(\frac{3}{2j})=1$, and $\displaystyle \int_{\frac{1}{j}}^{\frac{3}{2j}} h_jdr\le \frac{C}{j}$. If $g_j=-dt^2+f_j^2\sigma$, $g_0=-dt^2+\sigma$ then we find $\hat{d}_{t,g_j}\not\rightarrow \hat{d}_{t,g_0}$. Furthermore, we let $N=([0,1]\times \mathbb{D}^n,\hat{d}_{t,g_0})$, $L=([0,1]\times [0,1],d_{\text{taxi}}^{\lambda})$, 
      \begin{align}
       d_{\text{taxi}}^{\lambda}((s_1,r_1),(s_2,r_2))=|s_1-s_2|+\lambda|r_1-r_2|,
      \end{align}
      and 
      \begin{align}
         F: [0,1]\times \{1\}\subset L \rightarrow  [0,1] \times \{0\}\subset N 
      \end{align} 
      so that $F(t,1)=(t,0)$. Then if $(N\sqcup L, \bar{d})$ is the disjoint union of $N$ and $L$ we can define the metric space $(P,d_0)=(N \sqcup L, \bar{d}/ \sim)$ where we identify points by the map $F$ in order to conclude that 
      \begin{align}
          ([0,1]\times \mathbb{D}^n,\hat{d}_{t,g_j}) \rightarrow (P,d_0),
      \end{align}
      in the Gromov-Hausdorff sense. Furthermore, we find
      \begin{align}
          ([0,1]\times \mathbb{D}^n,\hat{d}_{t,g_j}) \rightarrow ([0,1]\times \mathbb{D}^n,\hat{d}_{t,g_0}),
      \end{align}
      in the Sormani-Wenger Intrinsic Flat sense.
  \end{ex}  
  \begin{proof}
     The proof of the Gromov-Hausdorff convergence was given in Example 3.6 of \cite{Allen-Null}. Here we will verify the Sormani-Wenger Intrinsic Flat convergence by applying Theorem \ref{thm-MainTheorem SWIF}. The volume convergence, area bound, and diameter bound of $f_j^2\sigma$ was established in Example 3.7 of \cite{Allen-Sormani-2} and hence we may apply Theorem \ref{thm-MainTheorem SWIF} to conclude SWIF convergence.
  \end{proof}

  \section{Estimating the SWIF Distance Between Spacetimes}\label{sect-Estimating SWIF}

  Now we explain how to construct a metric space $(Z,d_Z)$ in which we can isometrically embed two spacetimes equipped with the null distance. 
We will use this metric space to calculate the flat distance between the isometric images of two spacetimes $L_1$ and $L_2$. 
This will give us an upper bound on the intrinsic flat distance $d_{\mathcal F} (L_1, L_2)$. One should notice that the results of this section work for any globally hyperbolic spacetime and any time function which is differentiable  with respect to $t$. Hence we are providing the theoretical tools necessary to estimate the intrinsic flat distance required to address Conjecture \ref{conj-Conjecture SWIF}.

\begin{defn}\label{defn-Z}
Let $M$ be a compact connected manifold, $L=M \times [t_0,t_1]$, and $W \subset L$ a closed set. Let
\begin{align}
Z : =  \left(   L \times [0,H] \right) \sqcup  (L\times \{H+1\})|_\sim
\end{align}
where we identify $(x,H) \sim (x,H+1)$ for all $x \in W$. Let $g_1=\sigma_1-h_1^2dt^2$, $g_2=\sigma_2-h_2^2dt^2$ be two Lorentzian metrics where $\sigma_1,\sigma_2$ are Riemannian metrics for each time $t \in [t_0,t_1]$.

Now we define a special class of curves on $Z$, denoted by $\mathcal{C}$, so that  the curve $\gamma:[0,T]\rightarrow Z$ is in $\mathcal{C}$ if for $z_1=(\ell_1,h_1),z_2=(\ell_2,h_2) \in Z$ we find $\gamma(s)=(\alpha(s),h(s))$, $s \in [0,T]$ where $h(s)\in[0,H]\cup\{H+1\}$, $h(0)=h_1$, $h(T)=h_2$, and $\alpha(s)\in L$ is any piecewise smooth curve joining $\ell_1$ to $\ell_2$ so that $\alpha(s)$ is a piecewise causal curve with respect to $g_1$ if $h(s)=0$ and $g_2$ if $h(s) \in (0,H]\cup \{H+1\}$. 

If we let $\tau:L \rightarrow \R$ be a differentiable time function then we can define the length function for any $\gamma \in \mathcal{C}$
\begin{align}
    L_Z(\gamma)=\int_0^T |d\tau(\alpha'(s))|+|h'(s)|ds,
    \end{align}
and the corresponding distance function $d_Z: Z \times Z \to [0, \infty)$  by
\begin{align}
d_Z(z_1, z_2) = \inf \{L_Z(\gamma): \gamma \in \mathcal{C}, \, \gamma(0)=z_1,\, \gamma(1)=z_2\}.
\end{align}

Lastly, define functions $\varphi_1: L \to Z$ and $\varphi_2: L \to Z$ by 
\begin{align}
\varphi_1(x) = & (x, 0) \\
\varphi_2(x) =   &
\begin{cases}
(x,H+1)   & x \notin W \\
(x, H)  &  \textrm{otherwise.} 
\end{cases}
\end{align}
\end{defn}

Now we give some estimates on the metric space $(Z,d_Z)$ which will allow us to show that $\varphi_1$ and $\varphi_2$
isometrically embed $(L,\hat{d}_{t,g_1})$ and $(L,\hat{d}_{t,g_2})$ into $Z$, respectively.

\begin{lem}\label{lem-Z-Estimate} 
For $(Z, d_Z), \tau$ as in Definition \ref{defn-Z}, $(Z,d_Z)$ is a complete metric space and for all  $(\ell,h),(\ell',h') \in L\times[0,H] \subset Z$,
\begin{align}\label{dZEstimatetog_0}
d_Z((\ell,h),(\ell',h'))\ge \hat{d}_{\tau,g_1}(\ell,\ell') +|h-h'|
\end{align}
and
\begin{align} \label{region-dist-dec-to-Z}
d_Z((\ell,h),(\ell',h')) \le \hat{d}_{\tau,g_2}(\ell,\ell')+|h-h'|.
\end{align}
Furthermore, if  $\diam(L,\hat{d}_{\tau,g_2}) \le D$, 
\begin{align}\label{Z-Metric-Inequality-Assumption 1}
    \frac{\sigma_2(v,v)}{h_2^2} &\ge  \frac{\sigma_{1}(v,v)}{h_{1}^2}, \quad \forall (p,t) \in L, v \in T_pM, 
\end{align}
and  for all $x,y \in W$, $W \subset L$ closed, we assume that
\begin{align}\label{eq-distCond0}
\hat{d}_{\tau,g_2}(x,y) \le \hat{d}_{\tau,g_1}(x,y) +2 \delta
\end{align}
for some $\delta>0$ and we choose $H \ge \delta$, 
then $\varphi_1: (L,\hat{d}_{\tau,g_1}) \to (Z,d_Z)$ and $\varphi_2: (L,\hat{d}_{\tau,g_2}) \to (Z,d_Z)$ of Definition \ref{defn-Z} are distance preserving. 
\end{lem}

We note that $H$ is chosen to prevent having shorter paths  between pairs of points either in $(L,\hat{d}_{\tau,g_1})$ or $(L,\hat{d}_{\tau,g_2})$ seen as subsets of $Z$ than the ones in $(L,\hat{d}_{\tau,g_1})$ or $(L,\hat{d}_{\tau,g_2})$, themselves.

\begin{proof}
Let $C(s)=(\gamma(s),h(s))$, $s \in [0,1]$, be a curve connecting $(\ell,h),(\ell',h') \in L \times [0,H]\subset Z$ where it is enough to assume that $C(s) \subset L \times [0,H]$. If we let $\mathcal{C}_1$ be the set of all curves $C$ so that $\gamma$ is piecewise causal with respect to $g_1$,  $\mathcal{C}_2$ be the set of all curves $C$ so that $\gamma$ is piecewise causal with respect to $g_2$, and $\mathcal{C}$ be the set of all curves $C$ so that $\gamma$ is piecewise causal with respect to $g_1$ if $h(s)=0$ and $\gamma$ is piecewise causal with respect to $g_2$ if $h(s) \in (0,H]$ then by \eqref{Z-Metric-Inequality-Assumption 1} we see that
\begin{align}\label{eq-Curve Class Comparison}
    \mathcal{C}_2 \subset \mathcal{C} \subset \mathcal{C}_1.
\end{align}
Hence by taking infimums over these three classes of curves with respect to the length $L_Z$ defined in Definition \ref{defn-Z} we find \eqref{dZEstimatetog_0} and \eqref{region-dist-dec-to-Z}.

By \eqref{dZEstimatetog_0} we see that $\varphi_1(L)$ is isometrically embedded in $(Z,d_Z)$ since by \eqref{Z-Metric-Inequality-Assumption 1} it will always be shorter to consider a curve connecting points in $\varphi_1(L)$ which lies completely inside $\varphi_1(L)$.

 Now we would like to show that $\varphi_2(L)$ is isometrically embedded in $(Z,d_Z)$. By choosing $H\ge \delta$ we will show that for points in $W \times \{H\}$ it is never more efficient to take advantage of shortcuts in $L \times \{0\}\subset Z$.  To see this consider $C$ connecting $(p,H)$ to $(p',0)$ to $(q',0)$ and then to $(q,H)$. We see that by \eqref{dZEstimatetog_0}, \eqref{region-dist-dec-to-Z}, and \eqref{eq-Curve Class Comparison} that
 \begin{align}
     L_Z(C)&\ge \hat{d}_{\tau,g_2}(p,p')+H+\hat{d}_{\tau,g_1}(p',q')+\hat{d}_{\tau,g_2}(q',q)+H
     \\&\ge \hat{d}_{\tau,g_1}(p,p')+\hat{d}_{\tau,g_1}(p',q')+\hat{d}_{\tau,g_1}(q',q)+2H
     \\&\ge \hat{d}_{\tau,g_1}(p,q)+2H 
     \\&\ge\hat{d}_{\tau,g_1}(p,q)+2\delta \ge \hat{d}_{\tau,g_2}(p,q),\label{eq-Last Distance Estimate Z}
 \end{align}
 where we used \eqref{eq-distCond0} in \eqref{eq-Last Distance Estimate Z}. Hence for any curve $C$ connecting  points $(p,H),(q,H) \in W \times \{H\}$ we find
 \begin{align}
 L_{Z}(C) \ge \hat{d}_{\tau,g_2}(p,q).\label{LengthDistanceInequality1}
 \end{align}

Then by the fact that points in $(L\setminus W,H+1)$ are not glued to $L \times \{H\}$, and hence must enter $W\times \{H\}$ before attempting to take advantage of shortcuts in $L \times\{0\}\subset Z$, we are able to conclude that for points $(p,H+1),(q,H+1) \in (L \setminus W,H+1)$ or $(p,H+1) \in (L \setminus W,H+1)$, $(q,H) \in (W,H)$ or $(p,H) \in (W,H+1)$, $(q,H+1) \in (L \setminus W,H+1)$ and curves $C$ connecting them
\begin{align}
 L_{Z}(C) \ge \hat{d}_{\tau,g_2}(p,q).\label{LengthDistanceInequality2}
 \end{align}

Now for $p,q \in L$, by taking a curve $C \subset \varphi_2(L)$ connecting $\varphi_2(p),\varphi_2(q)$ whose length is within $\varepsilon>0$ of the distance $\hat{d}_{\tau,g_2}(p,q)$ and then taking $\varepsilon\rightarrow 0$ we can combine with \eqref{LengthDistanceInequality1} and \eqref{LengthDistanceInequality2} to find
\begin{align}
d_Z(   \varphi_2(p),\varphi_2(q) ) = \hat{d}_{\tau,g_2}(p,q).
\end{align}
Hence we can conclude that $\varphi_2(L)$ is isometrically embedded in $(Z,d_Z)$, as desired.
\end{proof}
  
%%%%%%%%%%%%%%%%%%%%%%%%
%%%%%%%%%%%%%%%%%%%%%%%%
%%%%%%%%%%%%%%%%%%%%%%%%

Now we can calculate the flat distance between $\varphi_{1_\sharp}[[L]]$ and $\varphi_{2_\sharp}[[L]]$. 

\begin{thm}\label{thm-Lorentzian SWIF Estimate}
Let $M$ be a compact, connected, and oriented manifold, $L=M \times [t_0,t_1]$, and $\tau:L\rightarrow \R$ a time function which is differentiable with respect to time. Let $g_1=\sigma_1-h_1^2dt^2$ and $g_2=\sigma_2-h_2^2dt^2$ be two Lorentzian metrics so that $\diam(L,\hat{d}_{\tau,g_2}) \le D$, 
\begin{align}\label{Z-Metric-Inequality-Assumption}
    \frac{\sigma_2(v,v)}{h_2^2} &\ge  \frac{\sigma_1(v,v)}{h_{1}^2}, \quad \forall (p,t) \in L, v \in T_pM.
\end{align}
If $\mathcal{H}_{\hat{d}_{\tau,g_2}}^{n+1}(L)\le V$, $\mathcal{H}_{\hat{d}_{\tau,g_2}}^{n}(\partial L) \leq A $, $W \subset L$ is a closed set,
\begin{align}\label{eq-volCond}
\mathcal{H}_{\hat{d}_{\tau,g_2}}^{n+1}( M \setminus W) \le V'
\end{align}
and assume that there exists a $\delta > 0$ so that for all $ x,y \in W$,
 \begin{align}\label{eq-distCond}
\hat{d}_{\tau,g_2}(x,y) \le \hat{d}_{\tau,g_1}( x,  y) +2 \delta
\end{align}
and that $H \ge \delta$.
Then
\begin{align}\label{Fest}
d^Z_{F}( \varphi_1(L), \varphi_2(L)) \le 2C_{n+1}V' + C_{n+2}H V + C_{n+1} H A 
\end{align}
where $\varphi_1,\varphi_2$ are the maps, and $Z$ is the metric space described in Definition \ref{defn-Z}.
\end{thm}

\begin{proof}
Apply Lemma  \ref{defn-Z} with $H \ge \delta$  to get a metric space $(Z,d_Z)$ and distance preserving maps $\varphi_1: L \to Z$ and $\varphi_2 : L  \to Z$.   
Our goal will be to define integral currents $T \in  \intcur_{m+1}(Z) $ and $T' \in  \intcur_m(Z)$ such that 
\begin{align}
\varphi_{2\#}[[L]]- \varphi_{1\#}[[ L]]   & =  \partial T + T', \\
\mass(T)    & \leq  C_{n+2}H V, \\ 
 \mass(T')   & \leq  C_{n+1}(2V'  +H A).
\end{align}
Then by the definition of flat convergence 
\begin{align}
d^Z_{F}( \varphi_2(L), \varphi_1( L)) \le  & \mass(T) + \mass(T')
\end{align}
and by the mass estimates we will find (\ref{Fest}).
   
To this end, since $\varphi_1$ and $\varphi_2$ are distance preserving maps,  
\begin{eqnarray}
\varphi_{1\#}[[ L_1]]  &  =  &  [[L \times \{0\}]], \\
\varphi_{2\#}[[ L_2]]  &  =  &  [[W  \times \{H \}  ]]  +  [[L \setminus W\times\{H+1\}]], 
\end{eqnarray}
where we are using the notation for a current on a globally hyperbolic spacetime as introduced in Example \eqref{ex-Globally Hyperbolic Integral Current Space}.
Then define
\begin{align}
T=  & [[  \,L \times [0,H] \,]]   \in  \intcur_{m+1}(Z),\label{eq-DefofT} \\ 
T'   =  & [[L \setminus W \times \{H+1\}]]  -    [[   (L \setminus W  ) \times \{H \} ]]  ,\label{eq-DefofT'1}
\\&\quad- [ [\, \partial L\times [0,H] \,]]    \in  \intcur_{m}(Z), \label{eq-DefofT'2}
\end{align}
where we are using standard notation for the product of integral current spaces.
If we compute the boundary
\begin{align}\label{eq-DefofBoundaryofT}
\partial T = &  [[ L  \times \{H \}  ]]  -   [[L\times \{0\}]]  +  [ [\, \partial L \times [0,H] \,]].
\end{align}
then we can conclude by \eqref{eq-DefofT}, \eqref{eq-DefofT'1}, \eqref{eq-DefofT'2}, and \eqref{eq-DefofBoundaryofT} that 
\begin{align}
\varphi_{2\#}[[L]]- \varphi_{1\#}[[ L]] = \partial  T + T'. 
\end{align}
Now by the definition of $T$ and $T'$, if we define for $(x,h),(x',h') \in Z$ the metric
\begin{align}
d_{Z'}((x,h),(x',h'))= \hat{d}_{\tau,g_2}(x,x') +|h-h'|
\end{align}
then we can calculate
\begin{align}
\mass(T) \leq  & C_{n+2} \mathcal{H}_{d_Z}^{n+2}(L\times [0,H])\\
\leq  & C_{n+2} \mathcal{H}_{d_{Z'}}^{n+2}(L\times [0,H])\\
\leq  & C_{n+2} H\mathcal{H}_{\hat{d}_{\tau,g_2}}^{n+1}(L),
\end{align}
\begin{align}
\mass(T')  &\leq  C_{n+1} \mathcal{H}_{d_Z}^{n+1}(L \setminus W \times \{H+1\})
\\&\qquad+C_{n+1} \mathcal{H}_{d_Z}^{n+1}((L \setminus W  ) \times \{H \})
\\&\qquad +C_{n+1} \mathcal{H}_{d_Z}^{n+1}(\partial L\times [0,H])
\\&\leq  C_{n+1} \mathcal{H}_{d_{Z'}}^{n+1}(L \setminus W \times \{H+1\})
\\&\qquad+C_{n+1} \mathcal{H}_{d_{Z'}}^{n+1}((L \setminus W  ) \times \{H \})
\\&\qquad +C_{n+1} \mathcal{H}_{d_{Z'}}^{n+1}(\partial L\times [0,H])
\\&\le 2C_{n+1}\mathcal{H}_{\hat{d}_{\tau,g_2}}^{n+1}(L \setminus W )+C_{n+1} H\mathcal{H}_{\hat{d}_{\tau,g_2}}^{n}(\partial L).
\end{align}
The desired estimate then follows by using the assumed Hausdorff measure bounds.
\end{proof}

\section{Proof of Main Theorems}\label{sect-Proofs}

In this section we provide the proofs of the main theorems. We start with a theorem which establishes uniform convergence from below. One should note that the proof given here is similar to the proof given for the corresponding result for warped products, Theorem 4.1 of \cite{Allen-Null}.

\begin{thm}\label{thm-LowerDistanceBound}
Let $M^n$ be a compact, connected manifold, $\sigma_j$ a sequence of continuous Riemannian manifolds on $M$, $L=[t_0,t_1]\times \Sigma$, $h_j:M\rightarrow (0,\infty)$, and $g_j=-h_j^2dt^2+\sigma_j$, $j \in \N\cup\{\infty\}$. Then if  
\begin{align}
    \frac{\sigma_j(x)(v,v)}{h_j^2} &\ge \left(1-\frac{1}{j}\right) \frac{\sigma_{\infty}(x)(v,v)}{h_{\infty}^2},\quad \forall x \in M, v \in T_xM. \label{eq-Lower Bound Assumption}
\end{align}
then
\begin{align}
    \hat{d}_{t,g_j} \ge  \hat{d}_{t,g_{\infty}}-C(j),
\end{align}
where $C(j) \ge 0$ and $C(j) \rightarrow 0$ as $j \rightarrow \infty$.
\end{thm}
\begin{proof}
First we notice that if we define $\bar{g}_j=-dt^2+\frac{\sigma_j}{h_j^2}$ then by the conformal invariance of the null distance we see that
\begin{align}\label{eq-Conformal Invariance Consequence}
    \hat{d}_{t,g_j}=\hat{d}_{t,\bar{g}_j}.
\end{align}

Define the metric $\tilde{g}_{j,\infty}=-dt^2+\left(1-\frac{1}{j}\right) \frac{\sigma_{\infty}}{h_{\infty}}$ and notice that by \eqref{eq-Lower Bound Assumption} every piecewise causal curve with respect to $\bar{g}_j$ is a piecewise causal curve with respect to $\tilde{g}_{j,\infty}$. Hence, since the null distance is defined as the infimum of the length of all piecewise causal curves and the length only depends on the time function we find that
\begin{align}
    \hat{d}_{t,\bar{g}_j}(p,q) \ge \hat{d}_{t,\tilde{g}_{j,\infty}}(p,q), \quad \forall p,q \in L.
\end{align}

Then we notice that $\tilde{g}_{j,\infty} \rightarrow \bar{g}_{\infty}$ uniformly so by Theorem 1.4 of \cite{AB}, we see that
\begin{align}
    \hat{d}_{t,\tilde{g}_{j,\infty}} \rightarrow \hat{d}_{t,\bar{g}_{\infty}},
\end{align}
uniformly, which implies the desired result by the conformal invariance observation in \eqref{eq-Conformal Invariance Consequence}
\end{proof}

The following lemma was proved in \cite{AB} except for the fact that there the time coordinate was in all of $\R$. Here we just notice that restricting $t \in [t_0,t_1]$ does not change the conclusion and we slightly change the expression for the metric to highlight the distance function with respect to $\sigma$.

\begin{lem}[Lemma 4.4 of \cite{AB}]\label{lem-Generalized Products Distance}
    Let $M^n$ be a compact, connected manifold, $\sigma$ a continuous Riemannian manifold on $M$, $L=[t_0,t_1]\times M$, and $g=-dt^2+\sigma$. Then if $x,y \in M$, $t,s \in [t_0,t_1]$ we find 
    \begin{align}
        \hat{d}_{t,g}((t,x),(s,y))=d_{\sigma}(x,y)+\max\{0,|t-s|-d_{\sigma}(x,y)\}.
    \end{align}
\end{lem}
\begin{proof}
Fix $p=(t,x),q=(s,y) \in L$ and let $\gamma$ be a piecewise smooth curve connecting $x$ to $y$, parameterized with respect to $\sigma$ arc length. Without loss of generality we may assume that $t \le s$. Consider the null curve with respect to $g$ 
    \begin{align}
        \alpha(\tau)= \left(t+\int_0^{\tau} |\gamma'(\eta)|_{\sigma}d\eta, \gamma(\tau)\right),
    \end{align}
    on $\R\times M$ where we note that $\alpha$ may not fit inside of $L$. Let $\tau' \in [0,\infty)$ be such that $\alpha(\tau')=(t',y)$ for some $t'\in[t_0,\infty)$ and notice that $\tau'\ge d_{\sigma}(x,y)$. 
    
    If $t'\le s$ then $q \in I^+(p)$ and we can build a new curve
    \begin{align}
        \tilde{\alpha}(\tau)= \left(t+\frac{|s-t|}{|t'-t|}\int_0^{\tau} |\gamma'(\eta)|_{\sigma}d\eta, \gamma(\tau)\right),
    \end{align}
    so that 
    \begin{align}
      \hat{d}_{t,g}(p,q)&\le \hat{L}_{t,g}(\tilde{\alpha}) 
      = |t-s| = d_{\sigma}(x,y)+\max\{0,  |t-s|-d_{\sigma}(x,y)\}.
    \end{align}
    If $t' > s$ and $\alpha$ fits inside of $L$ then we we can modify $\alpha$ by breaking into a piecewise null curve with two pieces. If $\alpha$ does not fit inside $L$ then we can modify $\alpha$ to be a piecewise null curve $\tilde{\alpha}$, with $k$ pieces. In either case this can be done so that on each piece defined on $[\tau_i,\tau_{i+1}]$ where $0\le\tau_i<\tau_{i+1}\le\tau'$ and $1\le i \le k$ where $\tilde{\alpha}$ is a null curve it is equal to
     \begin{align}
        \tilde{\alpha}(\tau)= \left(\tau_i\pm\int_0^{\tau} |\gamma'(\eta)|_{\sigma}d\eta, \gamma(\tau)\right).
    \end{align}
   To complete the proof we can follow the rest of the proof of Lemma 4.4 of \cite{AB}, possibly breaking up curves which do not fit into $L$ as was done above.
\end{proof}

We now prove a simple observation comparing the distance between points in $M$ to points in $L$ which will be useful for proving the main theorems of this paper.

\begin{lem}\label{lem-Reiammanian distance to Lorentzian Distance}
     Let $M^n$ be a compact, connected manifold, $\sigma_1,\sigma_2$ continuous Riemannian manifolds on $M$, $L=[t_0,t_1]\times M$, and $g_i=-dt^2+\sigma_i$, $i \in \{1,2\}$. Then if $x,y \in M$, $t,s \in [t_0,t_1]$ so that
    \begin{align}
        |d_{\sigma_1}(x,y)-d_{\sigma_2}(x,y)| <K, \quad \forall x, y \in W \subset M,
    \end{align}
then
\begin{align}
        |\hat{d}_{t,\sigma_1}((x,t),(y,s))-\hat{d}_{t,\sigma_2}((x,t),(y,s))| <2K, 
    \end{align}
    $\forall (x,t), (y,s) \in W \times [t_0,t_1]$.
\end{lem}
\begin{proof}
    We can calculate by Lemma \ref{lem-Generalized Products Distance}
    \begin{align}
        |\hat{d}_{t,\sigma_1}&((x,t),(y,s))-\hat{d}_{t,\sigma_2}((x,t),(y,s))| 
        \\&\le |d_{\sigma_1}(x,y)-d_{\sigma_2}(x,y)|
        \\& \quad + |\max\{0,  |t-s|-d_{\sigma_1}(x,y)\}-\max\{0,  |t-s|-d_{\sigma_2}(x,y)\}|,
    \end{align}
    and hence we need to bound the second term. 
    
    If both $d_{\sigma_1}(x,y),d_{\sigma_2}(x,y)\ge |t-s|$ then the second term is $0$ and we are done. If both $d_{\sigma_1}(x,y),d_{\sigma_2}(x,y)< |t-s|$ then 
    \begin{align}
    |\max&\{0,  |t-s|-d_{\sigma_1}(x,y)\}-\max\{0,  |t-s|-d_{\sigma_2}(x,y)\}| 
    \\&= |d_{\sigma_1}(x,y)-d_{\sigma_2}(x,y)| \le K,
    \end{align}
    and we are done. If $d_{\sigma_1}(x,y)< |t-s|\le d_{\sigma_2}(x,y) $ or $d_{\sigma_2}(x,y)< |t-s|\le d_{\sigma_1}(x,y) $ then for $i\in\{1,2\}$ equal to the index of the distance which is smaller we find
      \begin{align}
    |\max&\{0,  |t-s|-d_{\sigma_1}(x,y)\}-\max\{0,  |t-s|-d_{\sigma_2}(x,y)\}| 
    \\&= ||t-s|-d_{\sigma_i}(x,y)| 
    \\&\le |d_{\sigma_1}(x,y)-d_{\sigma_2}(x,y)|\le K.
    \end{align}
\end{proof}

We now make an observation that a H\"{o}lder bound between two Riemannian manifolds extends to a H\"{o}lder bound between the corresponding static spacetimes.

\begin{lem}\label{lem-Reiammanian Holder to Lorentzian Holder}
     Let $0<C\le C'$, $M^n$ be a compact, connected manifold, $\sigma_1,\sigma_2$ continuous Riemannian manifolds on $M$, $L=[t_0,t_1]\times M$, and $g_i=-dt^2+\sigma_i$, $i \in \{1,2\}$. Then if $x,y \in M$, $t,s \in [t_0,t_1]$, $\alpha \in (0,1]$ so that
    \begin{align}
        d_{\sigma_1}(x,y) <C d_{\sigma_2}(x,y)^{\alpha}, \quad \forall x, y \in M,
    \end{align}
then
\begin{align}
        \hat{d}_{t,\sigma_1}((t,x),(s,y)) <C' \hat{d}_{t,\sigma_2}((t,x),(s,y))^{\alpha}, 
    \end{align}
    $\forall (t,x), (s,y) \in  [t_0,t_1]\times M $.
\end{lem}
\begin{proof}   
    If both $d_{\sigma_1}(x,y),d_{\sigma_2}(x,y)\ge |t-s|$ then the second term in the distance estimate of Lemma \ref{lem-Generalized Products Distance} is $0$ and we are done. 
    If both $d_{\sigma_1}(x,y),d_{\sigma_2}(x,y)< |t-s|$ then by Lemma \ref{lem-Generalized Products Distance} we find
    \begin{align}
   \hat{d}_{t,\sigma_1}((t,x),(s,y))&=|t-s|
   \\\hat{d}_{t,\sigma_2}((t,x),(s,y))^{\alpha}&=|t-s|^{\alpha}
    \end{align}
    and since $|t-s|\le |t_0-t_1|$ we can pick a constant $C' \ge C$ so that the desired claim is true in this case. 
    If $d_{\sigma_1}(x,y)< |t-s|\le d_{\sigma_2}(x,y) $ then by Lemma \ref{lem-Generalized Products Distance} we find
      \begin{align}
   \hat{d}_{t,\sigma_1}((t,x),(s,y))&= |t-s|
   \\& \le C' |t-s|^{\alpha}
   \\&=C' d_{\sigma_2}(x,y){\alpha} = C'\hat{d}_{t,\sigma_2}((t,x),(s,y))^{\alpha}.
    \end{align}
    Lastly, if $d_{\sigma_2}(x,y)< |t-s|\le d_{\sigma_1}(x,y) $ then by Lemma \ref{lem-Generalized Products Distance} we find
    \begin{align}
   \hat{d}_{t,\sigma_1}((t,x),(s,y))&=d_{\sigma_1}(x,y)
   \\& \le Cd_{\sigma_2}(x,y)^{\alpha}
   \\& \le C|t-s|^{\alpha} = C\hat{d}_{t,\sigma_2}((t,x),(s,y))^{\alpha}.
    \end{align}
\end{proof}

In order to complete the proof of Theorem \ref{thm-MainTheorem GH} we will need to show pointwise convergence of the sequence of null distances as well as a H\"{o}lder distance bound from above. When combined with Lemmas \ref{lem-Generalized Products Distance}, \ref{lem-Reiammanian distance to Lorentzian Distance}, and \ref{lem-Reiammanian Holder to Lorentzian Holder} this will then imply uniform convergence of the null distances.

\begin{proof}[Proof of Theorem \ref{thm-MainTheorem GH}]
First notice that if $g_j=-h_{\alpha}^2 dt^2+\sigma_j$ and $\tilde{g}_j=-dt^2+\tilde{\sigma}_j$ where $\tilde{\sigma}_j=\frac{\sigma_j}{h_j^2}$ then by the conformal invariance of the null distance we have that
\begin{align}
    \hat{d}_{t,g_j}=\hat{d}_{t,\tilde{g}_j},
\end{align}
and hence we are justified in restricting to the case of $\tilde{g}_j$ for the rest of the argument.

Now we will show that the assumptions of Theorem \ref{thm-MainTheorem GH} imply compactness of the sequence of null distance functions of $\tilde{g}_j$ in the uniform topology. The first assumption of Theorem \ref{thm-MainTheorem GH} implies
\begin{align}
    \int_M|\tilde{\sigma}_j|_{\sigma_{\infty}}^{\frac{p}{2}}dV_{\sigma_{\infty}} \le C,
\end{align}
and hence by Theorem \ref{Thm-Allen-Holder} we find
\begin{align}
    d_{\tilde{\sigma}_j}(q_1,q_2) \le C(M,\sigma_{\infty}) d_{\sigma_{\infty}}(q_1,q_2)^{\frac{p-n}{p}},
\end{align}
for all $q_1,q_2 \in M$. Notice that this implies that there exists a $D \in (0,\infty)$ so that
\begin{align}
    \Diam(M,\tilde{\sigma}_j) \le D.
\end{align}
Now by Lemma \ref{lem-Reiammanian Holder to Lorentzian Holder} this implies there exists a $C' \ge C(M,\sigma_{\infty})$ so that
\begin{align}\label{eq-Upper Distance Bound}
    \hat{d}_{t,\tilde{g}_j}((s_1,q_1),(s_2,q_2)) \le C' \hat{d}_{t,g_{\infty}}((s_1,q_1),(s_2,q_2))^{\frac{p-n}{p}},
\end{align}
for all $(s_1,q_1),(s_2,q_2) \in L$. Then by combining \eqref{eq-Upper Distance Bound} with Theorem \ref{thm-LowerDistanceBound} we see by the Arzela-Ascoli Theorem that a subsequence must converge uniformly to a metric $d_{\infty}$ on $L$. We will denote the subsequence in the same way as the sequence. Our goal is to show that $d_{\infty}=\hat{d}_{t,\tilde{g}_{\infty}}$.

Next we will show that the assumptions of Theorem \ref{thm-MainTheorem GH} imply pointwise almost everywhere convergence of $(M,\tilde{\sigma}_j)$. By assumption we know that
\begin{align}
    \tilde{\sigma}_j &\ge \left(1-\frac{1}{j}\right)\tilde{\sigma}_{\infty},
    \\
    \Vol(M, \tilde{\sigma}_j) &\rightarrow \Vol(M,\tilde{\sigma}_{\infty}),
    \\\Area(\partial M, \tilde{\sigma}_j) &\le A,
    \\ \Diam(M, \tilde{\sigma}_j)& \le D,
\end{align}
and hence by Theorem \ref{thm-Allen-Bryden-VADB-II} we see that
\begin{align}
    d_{\tilde{\sigma}_j}(p,q) \rightarrow d_{\tilde{\sigma}_{\infty}}(p,q),
\end{align}
for almost every $p,q \in M$ with respect to $\sigma_{\infty}$.

Consider $p,q \in M$ so that
\begin{align}\label{Eq-Pointwise Distance Convergence}
    d_{\tilde{\sigma}_j}(p,q)\rightarrow d_{\tilde{\sigma}_{\infty}}(p,q).
\end{align}
Then by Lemma \ref{lem-Generalized Products Distance} we see that
\begin{align}
    \hat{d}_{t,\tilde{g}_j}((s_1,p),(s_2,q))&= d_{\tilde{\sigma}_j}(p,q)+\max\{0,|s_1-s_2|-d_{\tilde{\sigma}_j}(p,q)\},
    \\\hat{d}_{t,\tilde{g}_{\infty}}((s_1,p),(s_2,q))&= d_{\tilde{\sigma}_{\infty}}(p,q)+\max\{0,|s_1-s_2|-d_{\tilde{\sigma}_{\infty}}(p,q)\},
\end{align}
and hence \eqref{Eq-Pointwise Distance Convergence} implies that
\begin{align}
    \hat{d}_{t,\tilde{g}_j}((s_1,p),(s_2,q))\rightarrow \hat{d}_{t,\tilde{g}_{\infty}}((s_1,p),(s_2,q)).
\end{align}
Since the assumptions of Theorem \ref{thm-MainTheorem GH} imply pointwise convergence of $d_{\tilde{\sigma}_j}$ to $d_{\tilde{\sigma}_{\infty}}$ for almost every $p,q \in M$ we see by \eqref{Eq-Pointwise Distance Convergence} that $\hat{d}_{t,\tilde{g}_j}$ converges to $\hat{d}_{t,\tilde{g}_{\infty}}$ for almost all $(s_1,p),(s_2,q) \in L$. 

When we combine the pointwise almost everywhere convergence of  $\hat{d}_{t,\tilde{g}_j}$  to $\hat{d}_{t,\tilde{g}_{\infty}}$ with the uniform convergence of $\hat{d}_{t,\tilde{g}_j}$ to $d_{\infty}$ we see that the subsequence $\hat{d}_{t,\tilde{g}_j}$ uniformly converges to $\hat{d}_{t,\tilde{g}_{\infty}}$. Since this is true of every subsequence obtained by compactness we see that the original sequence must convergence uniformly, as desired.
\end{proof}

We now give the proof of Theorem \ref{thm-MainTheorem SWIF} which will apply the Sormani-Wenger Intrinsic Flat estimates of Theorem \ref{thm-Lorentzian SWIF Estimate}, the previous estimates for Riemannian manifolds of Theorem \ref{thm-Uniform Distance Bounds on W}, and Lemmas \ref{lem-Generalized Products Distance}, \ref{lem-Reiammanian distance to Lorentzian Distance}, and \ref{lem-Reiammanian Holder to Lorentzian Holder}

\begin{proof}[Proof of Theorem \ref{thm-MainTheorem SWIF}]

First notice that if $g_j=-h_j^2 dt^2+\sigma_j$ and $\tilde{g}_j=-dt^2+\tilde{\sigma}_j$ where $\tilde{\sigma}_j=\frac{\sigma_j}{h_j^2}$ then by the conformal invariance of the null distance we have that
\begin{align}
    \hat{d}_{t,g_j}=\hat{d}_{t,\tilde{g}_j},
\end{align}
and hence we are justified in restricting to the case of $\tilde{g}_j$ for the rest of the argument.
By assumption we know that if we rescale $\hat{\sigma}_j = \left(1-\frac{1}{j}\right)^{-1} \tilde{\sigma}_j$ then we find
\begin{align}
    \hat{\sigma}_j &\ge\tilde{\sigma}_{\infty},
    \\
    \Vol(M, \hat{\sigma}_j) &\rightarrow \Vol(M,\tilde{\sigma}_{\infty}),
    \\ \Diam(M, \hat{\sigma}_j)& \le D,
\end{align}
and hence by Theorem \ref{thm-Uniform Distance Bounds on W} that for any $\lambda  \in (0, \diam(M, \tilde{\sigma}_{\infty}))$ and  $\kappa >1$, there exists 
 a measurable set $W_{\lambda,\kappa} \subset M$ and a constant $\delta_{\lambda, \kappa,j}>0$ such that for all $p_1, p_2 \in W_{\lambda, \kappa}$
\begin{align}\label{eq-Riemannian Distance Estimate}
|d_{\hat{\sigma}_j}(p_1,p_2)-d_{\tilde{\sigma}_0}(p_1,p_2)| < 2 \lambda + 2\delta_{\lambda,\kappa,j},
\end{align}
where $\delta_{\lambda,\kappa,j} \to 0$ as $j \to \infty$,
and
\begin{align}\label{eq-Riemannian Volume Estimate}
\Vol(M \setminus W_{\lambda, \kappa},\hat{\sigma}_j) \le  \frac{1}{\kappa}\Vol(M,\tilde{\sigma}_0)+|\Vol(M,\hat{\sigma}_j)-\Vol(M,\tilde{\sigma}_0)|.
\end{align}

By Lemma \ref{lem-Reiammanian distance to Lorentzian Distance} we see that \eqref{eq-Riemannian Distance Estimate} implies

\begin{align}\label{eq-Riemannian Distance Estimate 2}
|\hat{d}_{t,\hat{g}_j}((s_1,p_1),(s_2,p_2))-\hat{d}_{\tilde{g}_0}((s_1,p_1),(s_2,p_2))| < 4 \lambda + 4\delta_{\lambda,\kappa,j},
\end{align}
for all $(s_1, p_1),(s_2, p_2) \in [t_0,t_1]\times W_{\lambda, \kappa} $ where $\hat{g}_j = -dt^2+\hat{\sigma}_j$.

Now notice by Lemma \ref{lem-Generalized Products Distance} we see that  
\begin{align}\label{eq-Euclidean Overestimate}
\hat{d}_{t,\hat{g}_j}((s_1, p_1),(s_2, p_2)) &\le \sqrt{d_{\hat{\sigma}_j}(p_1,p_2)^2+|s_1-s_2|^2}
\\&:= d_{\mathbb{E}_{\hat{\sigma}_j}} ((s_1, p_1),(s_2, p_2)).  
\end{align} 
So by combining \eqref{eq-Euclidean Overestimate} with \eqref{eq-Riemannian Volume Estimate} we find
\begin{align}
    \mathcal{H}^{n+1}_{\hat{d}_{t,\hat{g}_j}}([t_0,t_1]\times M\setminus W_{\lambda, \kappa}) &\le \mathcal{H}^{n+1}_{d_{\mathbb{E}_{\hat{\sigma}_j}}}([t_0,t_1]\times M\setminus W_{\lambda, \kappa})
    \\& \le \mathcal{H}^{n}_{d_{\hat{\sigma}_j}}(M\setminus W_{\lambda, \kappa})\mathcal{H}^{1}([t_0,t_1]) 
    \\&\le \Vol(M\setminus W_{\lambda, \kappa},\hat{\sigma}_j)|t_1-t_0| 
    \\&\le \left(\frac{\Vol(M,\tilde{\sigma}_j)}{\kappa}+C_j\right)|t_1-t_0|,
\end{align}
where $C_j \searrow 0$ as $j \rightarrow \infty$.
Similarly we can compute the bounds
\begin{align}
    \mathcal{H}^{n+1}(L)&\le \mathcal{H}^{n+1}_{d_{\mathbb{E}_{\hat{\sigma}_j}}}([t_0,t_1]\times M ) 
    \\&\le \mathcal{H}^{n}_{d_{\hat{\sigma}_j}}(M)\mathcal{H}^{1}([t_0,t_1]) \le C\Vol(M,\tilde{\sigma}_{\infty}) |t_1-t_0|:=\bar{V},
    \\\mathcal{H}^{n}(\partial L)&=\mathcal{H}^{n}_{d_{\mathbb{E}_{\hat{\sigma}_j}}}( [t_0,t_1]\times \partial M )+ \mathcal{H}^{n}_{d_{\mathbb{E}_{\hat{\sigma}_j}}}( \{t_0,t_1\}\times M )
    \\&\le \mathcal{H}^{n-1}_{d_{\hat{\sigma}_j}}(\partial M)\mathcal{H}^{1}([t_0,t_1]) + 2\mathcal{H}^{n}_{d_{\hat{\sigma}_j}}(M)
    \\&\le A |t_1-t_0|+2\Vol(M,\tilde{\sigma}_{\infty})C:=\bar{A}.
\end{align}
This implies for $H = 4 \lambda + 4\delta_{\lambda,\kappa,j}$ that by applying Theorem \ref{thm-Lorentzian SWIF Estimate} with the previous estimates we find
\begin{align}\label{Flat Distance Estimate}
d^Z_{F}( \varphi_0(L), \varphi_j(L)) &\le 2C_{n+1}\left(\frac{\Vol(M,\tilde{\sigma}_0)}{\kappa}+C_j\right)|t_1-t_0| 
\\&\quad +\left(4 \lambda + 4\delta_{\lambda,\kappa,j}\right) \left(C_{n+2} \bar{V} + C_{n+1} \bar{A} \right)
\end{align}
where $\varphi_j,\varphi_0$ are the maps, and $Z$ is the metric space described in Definition \ref{defn-Z}. First by taking the limsup on both sides we find
\begin{align}\label{Flat Distance Estimate}
\limsup_{j\rightarrow \infty}d^Z_{F}( \varphi_0(L), \varphi_j(L)) &\le 2C_{n+1}\frac{\Vol(M,\tilde{\sigma}_0)}{\kappa}|t_1-t_0| 
\\&\quad +4 \lambda \left(C_{n+2}\bar{V} + C_{n+1}  \bar{A} \right).
\end{align}
Now by taking $\lambda\rightarrow 0$ and $\kappa \rightarrow \infty$ we see that
\begin{align}\label{Flat Distance Estimate}
0\le \limsup_{j\rightarrow \infty}d^Z_{F}( \varphi_0(L), \varphi_j(L)) &\le 0,
\end{align}
which implies that 
\begin{align}
    (L,\hat{d}_{t,\hat{g}_j}) \rightarrow (L,\hat{d}_{t,\tilde{g}_0}),
\end{align}
in the Sormani-Wenger Intrinsic Flat sense.

Now we notice by Lemma \ref{lem-Reiammanian distance to Lorentzian Distance} that 
\begin{align}
   | \hat{d}_{t,\hat{g}_j}((s_1, p_1),(s_2, p_2))-\hat{d}_{t,\tilde{g}_j}((s_1, p_1),(s_2, p_2))|&\le \frac{4}{j},
\end{align}
and by definition $\hat{\sigma}_j \ge \tilde{\sigma}_j$ and hence by Theorem \ref{thm-Lorentzian SWIF Estimate} we see that
\begin{align}
    d_{\mathcal{F}}((L,\hat{d}_{t,\hat{g}_j}),(L,\hat{d}_{t,\tilde{g}_j}))&\le C_j
\end{align}
where $C_j \searrow 0$ as $j \rightarrow \infty$. Now by the triangle inequality for the Sormani-Wenger intrinsic flat distance we find
\begin{align}
   d_{\mathcal{F}}((L,\hat{d}_{t,\hat{g}_j}),(L,\hat{d}_{t,\tilde{g}_0}))&\le  d_{\mathcal{F}}((L,\hat{d}_{t,\tilde{g}_j}),(L,\hat{d}_{t,\hat{g}_j}))
   \\&\quad +d_{\mathcal{F}}((L,\hat{d}_{t,\hat{g}_j}),(L,\hat{d}_{t,\tilde{g}_0})),
\end{align}
which implies that 
\begin{align}
    (L,\hat{d}_{t,\tilde{g}_j}) \rightarrow (L,\hat{d}_{t,\tilde{g}_0}),
\end{align}
in the Sormani-Wenger Intrinsic Flat sense, as desired.
% and by Lemma \ref{lem-Reiammanian Holder to Lorentzian Holder} that
% \begin{align}
% \hat{d}_{t,\tilde{g}_j}((p_1,s),(p_2,t)) \le  \hat{d}_{t,\hat{g}_j}((p_1,s),(p_2,t)) \le \hat{d}_{t,\tilde{g}_j}((p_1,s),(p_2,t))
% \end{align}

% \textcolor{red}{How do we transform back to $\hat{d}_{t,\tilde{g}_t}$?}
\end{proof}

\noindent
{\bf Conflict of Interest:}

The authors have no competing interests to declare that are relevant to the content of this article.

\noindent
{\bf Data Availability Statement:}

This manuscript has no associated data.

%%%%%%%%%%%%%%%%%%%%%%%%%%%%%%%%%%%%%%%%%%%%%%%%%%%%%%
\bibliography{bibliography}

% \bib, bibdiv, biblist are defined by the amsrefs package.
\begin{bibdiv}
\begin{biblist}

\bib{AB}{article}{
      author={Allen, Brian},
      author={Burtscher, Annegret},
       title={{Properties of the Null Distance and Spacetime Convergence}},
        date={2021},
        ISSN={1073-7928},
     journal={International Mathematics Research Notices},
      volume={2022},
      number={10},
       pages={7729\ndash 7808},
}

\bib{Allen-Bryden}{article}{
      author={Allen, Brian},
      author={Bryden, Edward},
       title={Volume above distance below with boundary {II}},
        date={2025},
     journal={Ann. Global Anal. Geom.},
      volume={67},
      number={2},
}

\bib{AGH}{article}{
      author={Andersson, L.},
      author={Galloway, G.J.},
      author={Howard, R.},
       title={The cosmological time function},
        date={1998},
     journal={Classical Quantum Gravity},
      volume={15},
      number={2},
       pages={309\ndash 322},
}

\bib{AK}{article}{
      author={Ambrosio, Luigi},
      author={Kirchheim, Bernd},
       title={Currents in metric spaces},
        date={2000},
        ISSN={0001-5962},
     journal={Acta Math.},
      volume={185},
      number={1},
       pages={1\ndash 80},
}

\bib{Allen-Null}{article}{
      author={Allen, B.},
       title={Null distance and gromov-hausdorff convergence of warped product
  spacetimes},
        date={2023},
     journal={General Relativity and Gravitation},
      volume={55},
      number={118},
}

\bib{Allen-Holder}{article}{
      author={Allen, Brian},
       title={From {$L^p$} bounds to {G}romov-{H}ausdorff convergence of
  {R}iemannian manifolds},
        date={2024},
     journal={Geom. Dedicata},
      volume={218},
      number={2},
       pages={Paper No. 30, 22},
}

\bib{Allen-Perales}{article}{
      author={Allen, Brian},
      author={Perales, Raquel},
       title={Intrinsic flat stability of manifolds with boundary where volume
  converges and distance is bounded below},
        date={2025},
     journal={Comm. Anal. Geom.},
      volume={33},
      number={1},
       pages={185\ndash 238},
}

\bib{Allen-Perales-Sormani}{article}{
      author={Allen, Brian},
      author={Perales, Raquel},
      author={Sormani, Christina},
       title={Volume above distance below},
        date={2024},
     journal={J. Differential Geom.},
      volume={126},
      number={3},
       pages={837\ndash 874},
}

\bib{Allen-Sormani-2}{article}{
      author={Allen, Brian},
      author={Sormani, Christina},
       title={Relating notions of convergence in geometric analysis},
        date={2020},
     journal={Nonlinear Analysis},
      volume={200},
}

\bib{PBG}{article}{
      author={Burkhardt-Guim, Paula},
       title={Pointwise lower scalar curvature bounds for {$C^0$} metrics via
  regularizing {R}icci flow},
        date={2019},
     journal={Geom. Funct. Anal.},
      volume={29},
      number={6},
       pages={1703\ndash 1772},
}

\bib{BG2}{article}{
      author={Burtscher, Annegret},
      author={Garc\'ia-Heveling, Leonardo},
       title={Global hyperbolicity through the eyes of the null distance},
        date={2024},
     journal={Comm. Math. Phys.},
      volume={405},
      number={4},
       pages={Paper No. 90, 35},
}

\bib{BG}{article}{
      author={Burtscher, Annegret},
      author={García-Heveling, Leonardo},
       title={Time functions on lorentzian length spaces},
        date={2024},
     journal={Ann. Henri Poincar\'{e}},
}

\bib{Bernal-Sanchez}{article}{
      author={Bernal, Antonio~N.},
      author={S\'{a}nchez, Miguel},
       title={Smoothness of time functions and the metric splitting of globally
  hyperbolic spacetimes},
        date={2005},
        ISSN={0010-3616,1432-0916},
     journal={Comm. Math. Phys.},
      volume={257},
      number={1},
       pages={43\ndash 50},
}

\bib{BraunSamann}{misc}{
      author={Braun, Mathias},
      author={Sämann, Clemens},
       title={Gromov's reconstruction theorem and measured gromov-hausdorff
  convergence in lorentzian geometry},
        date={2025},
        note={arXiv:2506.10852[math.DG]},
}

\bib{FF}{article}{
      author={Federer, H.},
      author={Fleming, W.H.},
       title={Normal and integral currents},
        date={1960},
        ISSN={0003-486X},
     journal={Ann. of Math. (2)},
      volume={72},
       pages={458\ndash 520},
}

\bib{GG}{article}{
      author={Galloway, Gregory~J.},
       title={A note on null distance and causality encoding},
        date={2024},
     journal={Classical Quantum Gravity},
      volume={41},
      number={1},
       pages={Paper No. 017001, 5},
}

\bib{GS}{inproceedings}{
      author={Graf, Melanie},
      author={Sormani, Christina},
       title={Lorentzian area and volume estimates for integral mean curvature
  bounds},
        date={2022},
   booktitle={Developments in lorentzian geometry},
      editor={Albujer, Alma~L.},
      editor={Caballero, Magdalena},
      editor={Garc{\'i}a-Parrado, Alfonso},
      editor={Herrera, J{\'o}natan},
      editor={Rubio, Rafael},
   publisher={Springer International Publishing},
}

\bib{MNM}{article}{
      author={Hoseini, Mahla},
      author={Ebrahimi, Neda},
      author={Vatandoost, Mehdi},
       title={Cosmological time functions in {L}orentzian length spaces},
        date={2025},
        ISSN={1664-2368,1664-235X},
     journal={Anal. Math. Phys.},
      volume={15},
      number={4},
       pages={Paper No. 80},
}

\bib{HLP}{article}{
      author={Huang, Lan-Hsuan},
      author={Lee, Dan~A.},
      author={Perales, Raquel},
       title={Intrinsic flat convergence of points and applications to
  stability of the positive mass theorem},
        date={2022},
     journal={Ann. Henri Poincar\'e},
      volume={23},
      number={7},
       pages={2523\ndash 2543},
}

\bib{HLS}{article}{
      author={Huang, Lan-Hsuan},
      author={Lee, Dan~A.},
      author={Sormani, Christina},
       title={Intrinsic flat stability of the positive mass theorem for
  graphical hypersurfaces of {E}uclidean space},
        date={2017},
        ISSN={0075-4102},
     journal={J. Reine Angew. Math.},
      volume={727},
       pages={269\ndash 299},
         url={https://doi.org/10.1515/crelle-2015-0051},
      review={\MR{3652253}},
}

\bib{KCS}{article}{
      author={Kunzinger, M.},
      author={Sämann, C.},
       title={Lorentzian length spaces},
        date={2018},
     journal={Annals of Global Analysis and Geometry},
      volume={54},
}

\bib{KS}{article}{
      author={Kunzinger, Michael},
      author={Steinbauer, Roland},
       title={Null distance and convergence of lorentzian length spaces},
        date={2022},
     journal={Annales Henri Poincaré},
      volume={23},
}

\bib{M}{article}{
      author={M\"uller, Olaf},
       title={Lorentzian {G}romov-{H}ausdorff theory and finiteness results},
        date={2022},
     journal={Gen. Relativity Gravitation},
      volume={54},
      number={10},
       pages={Paper No. 117, 17},
}

\bib{MS}{article}{
      author={Minguzzi, E.},
      author={Suhr, S.},
       title={Lorentzian metric spaces and their {G}romov-{H}ausdorff
  convergence},
        date={2024},
     journal={Lett. Math. Phys.},
      volume={114},
      number={3},
       pages={Paper No. 73, 63},
}

\bib{MS-LorentzCompactness}{misc}{
      author={Mondino, Andrea},
      author={Sämann, Clemens},
       title={Lorentzian gromov-hausdorff convergence and pre-compactness},
        date={2025},
         url={https://arxiv.org/abs/2504.10380},
        note={arXiv:2504.10380[math.DG]},
}

\bib{MSS}{misc}{
      author={Meco, Benjamin},
      author={Sakovich, Anna},
      author={Sormani, Christina},
       title={Existence of uniform temple charts and applications to null
  distance},
        date={2025},
        note={arXiv:2509.11281[math.DG]},
}

\bib{N}{article}{
      author={Noldus, J.},
       title={A lorentzian gromov-hausdorff notion of distance},
        date={2004},
        ISSN={0264-9381},
     journal={Classical Quantum Gravity},
      volume={21},
      number={4},
       pages={839\ndash 850},
}

\bib{PGPHyp}{article}{
      author={Pacheco, Armando J.~Cabrera},
      author={Graf, Melanie},
      author={Perales, Raquel},
       title={Intrinsic flat stability of the positive mass theorem for
  asymptotically hyperbolic graphical manifolds},
        date={2023},
     journal={Gen. Relativity Gravitation},
      volume={55},
      number={11},
       pages={Paper No. 132, 45},
}

\bib{SS2}{misc}{
      author={Sakovich, Anna},
      author={Sormani, Christina},
       title={Introducing various notions of distances between space-times},
        date={2025},
         url={https://arxiv.org/abs/2410.16800},
        note={arXiv:2410.16800[math.DG]},
}

\bib{SS}{article}{
      author={Sakovich, A.},
      author={Sormani, C.},
       title={The null distance encodes causality},
        date={2023},
     journal={Journal of Mathematical Physics},
      volume={64},
      number={1},
}

\bib{SV}{article}{
      author={Sormani, C.},
      author={Vega, C.},
       title={Null distance on a spacetime},
        date={2016},
        ISSN={0264-9381},
     journal={Classical Quantum Gravity},
      volume={33},
      number={8},
       pages={085001, 29},
      review={\MR{3476515}},
}

\bib{Sormani-Wenger}{article}{
      author={Sormani, Christina},
      author={Wenger, Stefan},
       title={The intrinsic flat distance between {R}iemannian manifolds and
  other integral current spaces},
        date={2011},
        ISSN={0022-040X},
     journal={J. Differential Geom.},
      volume={87},
      number={1},
       pages={117\ndash 199},
}

\bib{V}{misc}{
      author={Vega, Carlos},
       title={Spacetime distances: an exploration},
        date={2021},
        note={arXiv:2103.01191[qr-qc]},
}

\end{biblist}
\end{bibdiv}

\end{document}